\documentclass[a4paper]{amsart}
\usepackage[english]{babel}

\usepackage{amssymb}
\usepackage{amsmath,amsthm}
\usepackage{mathtools}
\usepackage{dsfont}
\usepackage{nicefrac}
\usepackage{cases}
\usepackage{bm}
\usepackage{enumitem}

\usepackage[T1]{fontenc}
\usepackage{lmodern}
\usepackage{microtype}

\usepackage[a4paper, inner=3cm,outer=3cm,top=3.3cm,bottom=3cm]{geometry}

\usepackage{xspace}

\usepackage[autostyle]{csquotes}
\usepackage[dvipsnames]{xcolor}

\usepackage{tikz}
\usetikzlibrary{matrix,arrows,patterns,decorations.pathreplacing,svg.path}
\usepackage{pgfplots}

\usepackage{todonotes}

\usepackage[colorlinks, bookmarksdepth=3, citecolor=blue,linkcolor=blue]{hyperref}

\usepackage{pdfcomment}

\makeatletter
\@ifpackageloaded{hyperref}{%
\DeclareRobustCommand{\SkipTocEntry}[5]{}}{%
\DeclareRobustCommand{\SkipTocEntry}[4]{}}
\makeatother

\newenvironment{pf}{\proof[\proofname]}{\endproof}
\theoremstyle{plain}
\newtheorem{thm}{Theorem}[section]
\newtheorem{prop}[thm]{Proposition}
\newtheorem{cor}[thm]{Corollary}
\newtheorem{lem}[thm]{Lemma}

\newtheorem{defn}[thm]{Definition}
\numberwithin{equation}{section}
\numberwithin{figure}{section}
\theoremstyle{definition}
\newtheorem{ex}[thm]{Example}
\newtheorem{rem}[thm]{Remark}

\newcommand{\rk}{{\mathrm{rk}}}

\newcommand{\setcond}[2]{{\left\{ #1 \,:\, #2 \right\}}}

\DeclareMathOperator{\conv}{conv}

\newcommand{\Km}{{\mathcal{K}}}

\newcommand{\Pm}{{\mathcal{P}}}

\newcommand{\PMV}{{\mathcal{PMV}}}
\newcommand{\MV}{{\mathcal{MV}}}

\newcommand{\cal}[1]{\mathcal{#1}}
\newcommand{\C}{\mathbb C}

\newcommand{\Z}{\mathbb Z}
\newcommand{\R}{\mathbb R}
\newcommand{\Ss}{\mathbb S}

\newcommand{\N}{\mathbb N}

\newcommand{\G}{\Gamma}

\newcommand{\cI}{\cal I}

\newcommand{\cL}{\cal L}

\newcommand{\T}{{\mathbb T}}

\newcommand{\sig}{\sigma}

\DeclareMathOperator{\Pl}{P}
\DeclareMathOperator{\KG}{\mathbf K}

\DeclareMathOperator{\V}{V}

\DeclareMathOperator{\Vol}{Vol}
\DeclareMathOperator{\M}{M}
\DeclareMathOperator{\GL}{GL}
\DeclareMathOperator{\Sym}{\mathbf S}

\newcommand{\rs}[1]{Section~\ref{S:#1}}
\newcommand{\rl}[1]{Lemma~\ref{L:#1}}
\newcommand{\rp}[1]{Proposition~\ref{P:#1}}

\newcommand{\rex}[1]{Example~\ref{Ex:#1}}
\newcommand{\re}[1]{(\ref{e:#1})}
\newcommand{\rc}[1]{Corollary~\ref{C:#1}}
\newcommand{\rt}[1] {Theorem~\ref{T:#1}}

\newcommand{\rf}[1]{Figure~\ref{F:#1}}

\makeatletter
\@namedef{subjclassname@2020}{%
  \textup{2020} Mathematics Subject Classification}
\makeatother

\title[Pl\"ucker-type inequalities for mixed areas and intersection numbers of curves]{
Pl\"ucker-type inequalities for mixed areas and intersection numbers of curve arrangements}
\author{Gennadiy~Averkov}
\address{Fakult\"at 1, BTU Cottbus-Senftenberg, Platz der Deutschen Einheit 1, 03046 Cottbus, Germany}
\email{averkov@b-tu.de}
\author{Ivan~Soprunov}
\address{Department of Mathematics and Statistics, Cleveland State University,  2121 Euclid Ave, Cleveland, Ohio, 44115 USA}
\email{i.soprunov@csuohio.edu}

\begin{document}
\selectlanguage{english}

\date{}


\begin{abstract} Any collection of $n$ compact convex planar sets $K_1,\dots, K_n$ defines a vector of
${n\choose 2}$ mixed areas $\V(K_i,K_j)$ for $1\leq i<j\leq n$. We show that for $n\geq 4$
these numbers satisfy certain Pl\"ucker-type inequalities. Moreover, we prove 
that for $n=4$ these inequalities completely describe the space of all mixed area vectors
$(\V(K_i,K_j)\,:\,1\leq i<j\leq 4)$. For arbitrary $n\geq 4$ we show that this space has a semialgebraic closure
of full dimension. 
As an application, we show that the pairwise intersection numbers of any collection of $n$ tropical curves satisfy the 
Pl\"ucker-type inequalities. Moreover, in the case of four tropical curves, any homogeneous polynomial relation between their six intersection numbers follows from the corresponding Pl\"ucker-type inequalities.
\end{abstract}

\maketitle

\section{Introduction}

Let $K_1,\ldots,K_n \subset \R^d$ be a collection of nonempty compact convex sets and consider
their Minkowski linear combination with some non-negative real coefficients
$$\alpha_1 K_1+\dots+\alpha_nK_n := \setcond{ \alpha_1 p_1+\dots+\alpha_np_n }{p_i \in K_i}.$$ 

A fundamental result in the theory of convex sets asserts that the $d$-dimensional volume 
$\Vol_d( \alpha_1 K_1+\dots+\alpha_nK_n)$ 
 is a homogeneous polynomial in $\alpha_1,\ldots,\alpha_n$ of degree~$d$. This polynomial can be written as 
 \begin{equation}\label{e:mixed-polynomial}
 \Vol_d( \alpha_1 K_1+\dots+\alpha_nK_n)=\sum_{1 \le i_1,\ldots,i_n \le d}\! \V(K_{i_1},\ldots,K_{i_d})\,\alpha_{i_1} \cdots \alpha_{i_d},
 \end{equation} 
 where the coefficients $\V(K_{i_1},\ldots,K_{i_d})$ are chosen to be symmetric in the $d$ sets they depend on, see \cite[Th 5.1.7]{Schneider2014}.
 
 This gives rise to the notion of the mixed volume: a symmetric functional which sends every $d$-tuple of compact convex sets  $(K_1,\ldots,K_d)$ to
the coefficient $\V(K_1,\ldots,K_d)$. 
Mixed volumes are a unifying concept behind various metric or measure-related functionals considered in the theory of convex sets and
stochastic geometry, \cite{SantaloBook}.
Remarkably, they also appear in algebraic geometry due to the interpretation of the mixed volumes of lattice polytopes in terms of the number of solutions of sparse polynomial systems, provided by the celebrated Bernstein-Khovanskii-Kushnirenko (BKK) theorem \cite{Bernstein75, Kh78, Ku76}. More recently, mixed volumes 
have become one of the key objects in tropical intersection theory \cite{BB13,MaclaganSturmfels15}.

Understanding relationships between mixed volumes in terms of inequalities is a central topic in the theory of mixed volumes. While some of the inequalities, like the Aleksandrov-Fenchel inequality, have long been known and widely used, it appears that many inequalities
are yet to be discovered. Currently, there is no complete understanding of all relations between mixed volumes, besides the cases of small $n$ in 
dimension $d=2$.

The problem that we address in this paper goes back to the work of Heine \cite{heine1938wertvorrat} and
Shephard \cite{Shephard1960}, whereas the format of the problem is in the spirit of the so-called Blaschke-Santal\'o diagrams, \cite{Blaschke, Santalo61}. Having a family of functionals $\phi_1,\ldots,\phi_m : S \to \R$ on a certain space $S$, one can ask for
all relationships between these functionals, for example, by providing an inequality description for the set 
$D=\setcond{(\phi_1(s),\ldots,\phi_m(s))}{s \in S}$.
Santal\'o called $D$ the {\it diagram} for the functionals $\phi_1,\ldots,\phi_m$ and studied such 
diagrams for the space of planar compact convex sets and triples of functionals common in convexity theory (such as area, perimeter, diameter, width etc.), see \cite{Santalo61}. If an inequality description of $D$ can be established, one has a complete system of best-possible inequalities that link 
$\phi_1,\ldots,\phi_m$. 

Our ultimate goal is to understand diagrams of mixed volume functionals on the space of $n$-tuples $K=(K_1,\ldots,K_n)$ of compact convex sets in $\R^d$. There is a total of $\binom{n+d-1}{d}$ mixed volumes in \re{mixed-polynomial} obtained by choosing $d$ out of the $n$ convex sets
with repetitions allowed. We call the corresponding diagram $\MV(n,d)\subset \R^{\binom{n+d-1}{d}}$ the {\it mixed volume configuration space}.
If we do not allow repetitions, there are  $\binom{n}{d}$ mixed volumes which form, what we call, the 
{\it pure mixed volume configuration space} $\PMV(n,d)\subset \R^{\binom{n}{d}}$.

In this paper we focus on dimension $d=2$, which already offers many challenges. In this case, the elements of 
$\PMV(n,2)$ are the pure mixed areas $\V(K_i,K_j)$ for $1 \le i < j \le n$. Together with the usual areas 
$\Vol_2(K_i)=\V(K_i,K_i)$ for $1\leq i\leq n$ they form the space $\MV(n,2)$.
At the moment, a description of $\MV(n,2)$ in terms of inequalities is only known for $n \le 3$.
For $n=2$, it is not hard to see that $\MV(2,2)$ is described by the Minkowski inequality, see \rex{n=2}. In \cite{heine1938wertvorrat} Heine
gave a complete description for $\MV(3,2)$ in terms of the Minkowski inequalities and determinantal inequalities. All other cases remain open.
 
 
We summarize the main results of this paper below.

\begin{enumerate}
	\item We introduce new inequalities for mixed areas, which we call Plücker-type inequalities, as they resemble the Plücker relations for the Grassmannian. The distinctive feature of the Plücker-type inequalities is that they only involve pure mixed areas, that is, none of the mixed volume functionals contains a repetition of a set. To the best of our knowledge, no  inequalities  of such a form have been previously known.
	
	\item We prove that the Plücker-type inequalities provide a complete description of $\PMV(n,2)$ in terms of inequalities for $n=4$,
	but are not enough for $n\geq 8$. Whether they provide a complete description when $n=5,6,7$ is open.
	
	\item A priori, it is not clear whether $\MV(n,2)$ always has a description in algebraic terms, although the semialgebraicity of $\MV(n,2)$ looks like a plausible conjecture. For example, all known tight inequalities 
	in convexity always turn out to be polynomial inequalities. 
	While we leave the verification of the semialgebraicity of $\MV(n,2)$ as an open problem, we can at least confirm that
	$\PMV(n,2)$ (being a projection of $\MV(n,2)$) has the closure that is a semialgebraic set. 
	\item We show that the sets $\MV(n,2)$ and $\PMV(n,2)$ have full topological dimension. In particular, this means that there are no polynomial equations that the coefficients of the mixed volume polynomial  \re{mixed-polynomial} satisfy.
	\item According to the BKK theorem and its tropical analog, the (normalized) mixed area has an interpretation as the intersection number
	of two generic algebraic curves in $(\C\setminus\{0\})^2$ with given Newton polytopes, as well as the intersection number of two tropical curves intersecting transversally. 
	Our results have the following application in this setting. 
	Every collection of $n$ (tropical) curves defines a vector of ${n \choose 2}$ pairwise intersection numbers, so one can consider the
	configuration space $\cI(n,2)$ of all such vectors. We show that the smallest (with respect to inclusion) positively homogeneous closed
	set containing $\cI(n,2)$ is semialgebraic. Moreover, for $n=4$ it is described by three Plücker-type inequalities.
\end{enumerate} 

The scope of our studies can be naturally extended to studying $\MV(n,d)$ and $\PMV(n,d)$ in a general dimension $d$. This leads to even more challenging problems. For example, there is no complete description available
neither for $\PMV(5,3)$ nor for $\MV(3,3)$. The volume polynomial \re{mixed-polynomial} is a special case of a class of polynomials, called Lorenzian polynomials, which have been extensively studied in connection with algebraic geometry and algebraic combinatorics, see \cite{BH20}. This provides another motivation for studying polynomial relations between the coefficients of the volume polynomial. 

\subsection*{Acknowledgments} A large portion of this work was completed during the second author's visit to Germany in Fall 2021. 
He is grateful to the first author's family, as well as MathCoRe Magdeburg and BTU Cottbus faculty and staff, for their hospitality that made this visit productive and pleasant. Both authors are thankful to Chris Borger and Tobias Boege for many fruitful discussions. In particular, a big thank you to Tobias Boege for posing the question of dimension of the mixed volume configuration spaces and for constant interest in this work. We also thank Bernd Sturmfels for pointing out related research on Lorentzian polynomials. Our computer calculations were done using SageMath \cite{sagemath}.
\section{Preliminaries}

In this section we introduce the main objects of study  and provide some preliminary results.

Throughout the paper we use $[n]$ for the set $\{1,\dots,n\}$. 
Also, we will often index the elements of $\R^{n\choose 2}$ using a pair of indices, either as $v_{i,j}$ or, simply, $v_{ij}$ for all $\{i,j\}\subset [n]$.

Let $\Km_d$ be the set of all nonempty compact convex subsets of $\R^d$.
For $K\in\Km_d$ let $\Vol_d(K)$  denote the Euclidean $d$-dimensional volume of $K$. Given two subsets $A,B$ of $\R^d$, their {\it Minkowski sum} $A+B$ is the vector sum $A+B=\{a+b \,:\, a\in A, b\in B\}$. 

The {\it mixed volume} $\V(K_1,\dots, K_d)$ is the unique symmetric and multilinear (with respect to Minkowski addition) function on $d$ elements $K_1,\dots, K_d$ of $\Km_d$, which coincides with $\Vol_d$ on the diagonal, i.e.,
$\V(K,\dots, K)=\Vol_d(K)$ for any $K\in\Km_d$. One can write an explicit ``polarization'' formula for the mixed volume \cite[Sec 5.1]{Schneider2014}:
\begin{equation}\label{e:polarization}
\V(K_1,\dots,K_d)=\frac{1}{d!}\sum_{\ell=1}^d(-1)^{d+\ell}\!\sum_{i_1<\dots<i_\ell}\Vol_d(K_{i_1}+\dots+K_{i_\ell}).
\end{equation}

In particular, for $d=2$ we have 
\begin{equation}\label{e:polarization-2}
\V(K_1,K_2)=\frac{1}{2}\left(\Vol_2(K_1+K_2)-\Vol_2(K_1)-\Vol_2(K_2)\right).
\end{equation}

Some of the basic properties of the mixed volume we will need are the following, see \cite[Sec 5.1]{Schneider2014}: 
\begin{itemize}
\item[(a)] $\V(K_1,\dots, K_d)$ is non-negative,
\item[(b)] $\V(K_1,\dots, K_d)$ is monotone with respect to inclusion in each entry,
\item[(c)] $\V(K_1,\dots, K_d)$ is invariant under independent translation of each $K_i$,
\item[(d)] $\V(\phi(K_1),\dots, \phi(K_d))=|\det(\phi)|\V(K_1,\dots, K_d)$ for any linear transformation $\phi:\R^d\to\R^d$.
\end{itemize}

Let $K\in\Km_d$. Recall that the {\it support function} 
$h_K:\Ss^{d-1}\to\R$ of $K$ is defined by 
$$h_K(u)=\max\{\langle u, x\rangle\,:\,x \in K\},$$
where $\Ss^{d-1}$ is the $(d-1)$-dimensional unit sphere and $\langle u, x\rangle$  is the standard inner product in $\R^d$.
Consider a convex polytope $P\subset\R^d$.
Let $U_P$ be the set of the outer normals to the facets of $P$ and 
$P^{u}$ be the facet with the outer unit normal $u$. Then (see, for example, \cite[Eq. (5.2)]{Schneider2014})
\begin{equation}\label{e:vol-sum}
\Vol_d(P)=\frac{1}{d}\sum_{u\in U_P} h_P(u) \Vol_{d-1}(P^{u}).
\end{equation}
Geometrically, \re{vol-sum} expresses the volume of $P$ as the sum of the volumes of pyramids over the facets of $P$ with the same apex lying in the interior of $P$. There is a similar formula for the mixed volume $\V(K_1,P_2,\dots, P_d)$ where
$K_1\in\Km_d$ and $P_2,\dots, P_d$ are convex polytopes. 
Let $U_{P_2,\dots, P_d}$ be the set of the outer unit normals to the facets of $P_2+\dots+P_d$ and let
$P_i^{u}=\max\{\langle u, x\rangle\,:\,x \in P_i\}$ be the face of $P_i$ corresponding to $u$, for $2\leq i\leq d$.
Then 
\begin{equation}\label{e:mix-sum}
\V(K_1,P_2,\dots, P_{d})=\frac{1}{d}\!\sum_{u\in U_{P_2,\dots, P_d}}\!h_{K_1}(u)\V(P_2^{u},\dots, P_{d}^{u}),
\end{equation}
see, for example, \cite[Eq. (5.23)]{Schneider2014}.
For $d=2$, \re{mix-sum} becomes particularly simple:
\begin{equation}\label{e:mix-sum-2}
\V(K,P)=\frac{1}{2}\sum_{u\in U_P}h_{K}(u)|P^{u}|,
\end{equation}
where the sum is over the outer normals to the sides of $P$ and $|P^u|=\Vol_1(P^{u})$ is the length of the side~$P^{u}$.

We will need the following simple observations. We say a convex set $K$ is {\it inscribed} in a polygon $P$ if
$K\subseteq P$ and every side of $P$ has a nonempty intersection with $K$.

\begin{lem}\label{L:inscribed} 
Let $K,P\in\Km_2$ where $P$ is a convex polygon. If $K$ is inscribed in $P$ then $$\V(K,P)=\Vol_2(P).$$
\end{lem}

\begin{pf} Since $K$ is inscribed in $P$ we have $h_K(u)=h_P(u)$ for every outer unit normal $u$ to a side of $P$. The lemma now follows 
by comparing \re{vol-sum} for $d=2$ and \re{mix-sum-2}.
\end{pf}

Recall that the {\it width} of $K$ in the direction of unit vector $u$ 
is defined by  $$w_u(K)=\max\{\langle u,x\rangle:x\in K\}-\min\{\langle u,x\rangle:x\in K\}.$$
We have 

\begin{lem}\label{L:width} 
Let $K,I\in\Km_2$ where $I$ is a segment of length $\lambda$, orthogonal to $u\in\Ss^1$. Then $$\V(K,I)=\frac{1}{2} w_u(K)\lambda.$$
\end{lem}

\begin{pf} This follows directly from \re{mix-sum-2} since the outer normals of $I$ are $\{u,-u\}$ and
$h_{K}(u)=\max\{\langle u,x\rangle:x\in K\}$ and $h_{K}(-u)=\max\{\langle -u,x\rangle:x\in K\}=-\min\{\langle u,x\rangle:x\in K\}$.
\end{pf}

Now consider an $n$-tuple $K=(K_1,\dots, K_n)$ of sets in $\Km_d$. Then any choice of $n$ non-negative integers
$d_1,\dots, d_n$ satisfying $d_1+\dots+d_n=d$ gives rise to a value of the mixed volume
$$\V_K(d_1,\dots, d_n):=\V( \underbrace{K_1,\ldots,K_1}_{d_1},\ldots,\underbrace{K_n,\ldots,K_n}_{d_n}).$$
There are ${n+d-1\choose d}$ such choices which  produce the vector 
$$\V_K:=\big( \V_K(d_1,\dots, d_n)\,:\,d_1+\dots+d_n=d\, \big)\in\R^{d+n-1\choose d}$$
which we call the {\it mixed volume configuration vector} corresponding to $K=(K_1,\dots, K_n)$. For example, when $n=2$, $d=3$, and 
$K=(K_1,K_2)\in{\Km_3}^2$, we obtain the vector
$$\V_K=\Big( \V_K(3,0), \V_K(2,1), \V_K(1,2), \V_K(0,3)\Big)\in\R^4$$
where the first and the last entries are the 3-dimensional volumes of $K_1$ and $K_2$, and the middle ones are $\V(K_1,K_1,K_2)$ and $\V(K_1,K_2,K_2)$, respectively.

\begin{defn}\label{D:config}
The {\it mixed volume configuration space} is the set of all mixed volume configuration vectors over all possible $n$-tuples $K\in {\Km_d}^{\!n}$
$$\MV(n,d)=\left\{ \V_K\in\R^{d+n-1\choose d}\,:\,K\in {\Km_d}^{\!n}\right\}.$$
\end{defn}

\begin{ex}\label{Ex:n=2} Let $n=d=2$. According to the Minkowski inequality \cite[Th 7.2.1]{Schneider2014},
\begin{equation}\label{e:Mink}
\V(K_1,K_2)^2\geq \Vol_2(K_1)\Vol_2(K_2)
\end{equation} 
for any $K_1,K_2\in\Km_2$. This implies that 
$$\MV(2,2)\subseteq \left\{(v_{11},v_{12},v_{22})\in\R_{\geq 0}^3\,  :\,  v_{12}^2\geq v_{11}v_{22}\right\}.$$
It is not hard to see that the above inclusion is, in fact, equality. For this it is enough to 
take $K_1,K_2$ to be rectangular boxes, $K_1=[0,a]\times[0,b]$ and $K_2=[0,c]\times[0,d]$. In this case
$\Vol_2(K_1)=ab$, $\V(K_1,K_2)=\frac{1}{2}(ad+bc)$, and $\Vol_2(K_2)=cd$.
One can then check that for any non-negative $(v_{11},v_{12},v_{22})$ satisfying the Minkowski inequality, the system
$\{ ab=v_{11}, ad+bc=2v_{12}, cd=v_{22}\}$ reduces to a quadratic equation with discriminant  $v_{12}^2- v_{11}v_{22}$ and has a 
solution $(a,b,c,d)\in\R_{\geq 0}^4$.
%

\end{ex}




We next introduce the main object of this paper, the pure mixed volume configuration space $\PMV(n,d)$.

\begin{defn}\label{D:pure}  Given an $n$-tuple $(K_1,\dots, K_n)$ of convex sets in $\Km_d$, define the {\it pure mixed volume configuration vector} 
\begin{equation}\label{e:pure-config-vector}
W_K:=\big( \V(K_{i_1},\dots, K_{i_d})\,:\,\{i_1,\dots, i_d\}\subset [n] \big)\in\R^{n\choose d}.
\end{equation}
The {\it pure mixed volume configuration space} is the set of all pure mixed volume configuration vectors over all possible $n$-tuples $K\in {\Km_d}^n$
$$\mathcal{PMV}(n,d)=\left\{ W_K\in\R^{n\choose d}\,:\,K\in {\Km_d}^{\!n}\right\}.$$
\end{defn}
Note that $\mathcal{PMV}(n,d)$ is empty for $n<d$ and $\mathcal{PMV}(d,d)=\R_{\geq 0}$. Also, for $n\geq d$, $\mathcal{PMV}(n,d)$ is the projection of
$\MV(n,d)$ onto a coordinate subspace. 

\begin{ex}\label{Ex:PMV-3-2}
The first non-trivial case is the case of three planar convex sets, i.e. $n=3$ and $d=2$.
In this case $\mathcal{PMV}(3,2)=\R_{\geq 0}^3$. This can be realized using segments and points only. 
Indeed, consider a vector $(v_{12},v_{13},v_{23})\in\R_{\geq 0}^3$. We need to choose
$K_1,K_2,K_3$ in $\Km_2$ such that $v_{ij}=\V(K_i,K_j)$ for all $1\leq i<j\leq 3$.
 If all $v_{ij}=0$ we may take $K_1=K_2=K_3=\{0\}$. 
If $v_{12}>0$, $v_{13}\geq 0$ and $v_{23}=0$  we take
$K_1=[0,\sqrt{2}v_{12}e_1]$, $K_2=[0,\sqrt{2}e_2]$, and $K_3=[0,\sqrt{2}(v_{13}/v_{12})e_2]$. 
Finally, when all $v_{ij}>0$ we may take
$$K_1=\left[0,\frac{\sqrt{2v_{12}v_{13}v_{23}}}{v_{23}}e_1\right],\  K_2=\left[0,\frac{\sqrt{2v_{12}v_{13}v_{23}}}{v_{13}}e_2\right],\ K_3=\left[0,\frac{\sqrt{2v_{12}v_{13}v_{23}}}{v_{12}}(e_1+e_2)\right].$$
It is not hard to extend this argument to show that $\mathcal{PMV}(n,n-1)=\R_{\geq 0}^{n}$ for any $n\geq 3$. 
\end{ex}

\section{Statements of main results}

It is known that if $n\geq 2$ and $d\geq 2$ then $\MV(n,d)$ is a proper subset of $\R_{\geq 0}^{n+d-1\choose d}$ as there are many algebraic inequalities relating the mixed volumes (the entries of $\V_K$). One of the most important ones is the Aleksandrov-Fenchel inequality:
\begin{equation}\label{e:AF}\tag{AF}
\V(A,B,K_3,\dots, K_d)^2\geq \V(A,A,K_3,\dots, K_d)\V(B,B,K_3,\dots, K_d),
\end{equation}
for any $K_1,\dots,K_n\in\Km_d$, \cite[Ch 7]{Schneider2014}. It generalizes the Minkowski inequality \re{Mink}, as well as the classical isoperimetric and Brunn-Minkowski inequalities. However, starting with $n\geq 3$ these inequalities do not describe $\MV(n,d)$ completely, and more inequalities are needed, see for example,
\cite[Cor 4.2]{ABS20}. When we consider the pure mixed volume configuration space $\mathcal{PMV}(n,d)$ the Aleksandrov-Fenchel inequalities are absent,
as the mixed volumes in the right hand side of \re{AF} are not part of the coordinates of $W_K$.


As we saw in \rex{PMV-3-2}, for $n=3$ the space $\mathcal{PMV}(n,2)$ is the entire positive orthant $\R_{\geq 0}^3$.
However, starting with $n\geq 4$ it is a proper subset of the orthant $\R_{\geq 0}^{n\choose 2}$, as the mixed areas
satisfy certain inequalities,
which we call the Pl\"ucker-type inequalities. Below and throughout the paper $A\sqcup B$ denotes the disjoint union of sets $A$ and $B$.

\begin{defn}\label{D:PTE}
 The  {\it Pl\"ucker-type inequalities} on $\R_{\geq 0}^{n\choose 2}$ are
\begin{equation}\label{e:PTE}\tag{PT}
v_{ij}v_{kl}\leq v_{ik}v_{jl}+v_{il}v_{jk}\ \ \text{ for all }\ \{i,j\}\sqcup\{k,l\}\subseteq[n].
\end{equation}
\end{defn}
 For $n=4$ we have exactly three such inequalities, which one can also interpret as the triangle inequalities for the three products
$v_{12}v_{34}$, $v_{13}v_{24}$, and $v_{14}v_{23}$:
\begin{equation}\label{e:PTE-4}
\begin{matrix}
\,\,\, v_{12}v_{34}+v_{13}v_{24}-v_{14}v_{23}\geq 0\\
\,\, \, v_{12}v_{34}-v_{13}v_{24}+v_{14}v_{23}\geq 0\\
-v_{12}v_{34}+v_{13}v_{24}+v_{14}v_{23}\geq 0.
\end{matrix}
\end{equation}
Clearly,  the Pl\"ucker-type inequalities are invariant under the permutations of $\{1,2,3,4\}$. 

Now let $K_1,\dots, K_4$ be nonempty compact convex sets in $\R^2$ and let $V_{ij}$
denote the mixed area $\V(K_i,K_j)$. 
One of the key observations is that any linear combination of the form
$$\pm V_{12}V_{34}\pm V_{13}V_{24}\pm V_{14}V_{23}$$ is multilinear, that is it is Minkowski-additive 
 in each $K_i$ and is homogeneous with respect to rescaling each $K_i$ by a positive real factor. 
This property does not hold for expressions involving mixed volumes with repetitions, as in \re{AF}.

Here is the first main result of the paper.

\begin{thm}\label{T:PMV-4-2} 
We have
\begin{equation}\label{e:main}
\mathcal{PMV}(4,2)= \left\{v\in\R_{\geq 0}^{4\choose 2} \,:\, 
v_{ij}v_{kl}\leq v_{ik}v_{jl}+v_{il}v_{jk}\ \text{\rm for }\{i,j\}\sqcup\{k,l\}=[4] \right\}.
\end{equation}
\end{thm}
In particular, this shows that $\PMV(4,2)$ is a  basic closed semialgebraic subset of $\R^6$ (see \rs{PMV-n-2} for the definition of basic semialgebraic sets).


Here are the main steps of the proof of \rt{PMV-4-2}.
Recall that planar compact convex sets can be approximated by convex polygons in Hausdorff metric (see \cite[Th 1.8.16]{Schneider2014})
and the mixed volume functional is continuous (see the proof of Theorem 5.1.7 in \cite{Schneider2014}).
Thus, to prove the direct inclusion in \re{main} it is enough to restrict ourselves to the case of convex polygons $K_1,\dots, K_4$. Furthermore,
each convex polygon in the plane is the Minkowski sum of segments and triangles (see, for example, \cite[Th 3.2.14]{Schneider2014}). Since the Pl\"ucker-type inequalities are linear in each $K_i$ the problem further reduces to the case when each $K_i$ is either a segment or a triangle. We deal with this situation in \rs{PTE}.
To prove the reverse inclusion in \re{main} we first put a combinatorial structure on the space
defined by the Pl\"ucker-type inequalities. This allows us to handle the case distinction when some of the $v_{ij}$ become zero, similarly to 
what we did in \rex{PMV-3-2}. We do this in \rs{graphs}.

Our second main result gives an insight about the pure mixed volume configuration space for more than four planar convex sets.
Using the same arguments as in \rs{PTE}, 
we see that $\mathcal{PMV}(n,2)$ is contained 
in the semialgebraic subset of $\R^{n \choose2}$  described by the $3{n\choose 4}$ Pl\"ucker-type inequalities. 
However, for $n\geq 8$ more inequalities are needed to describe $\mathcal{PMV}(n,2)$ completely.
Whether the Pl\"ucker-type inequalities describe $\mathcal{PMV}(n,2)$ for $n=5,6,7$ remains open. We have the following theorem.

 \begin{thm}\label{T:PMV-n-2} 
For $n\geq 4$, we have
$$\PMV(n,2)\subseteq \left\{v\in\R_{\geq 0}^{n\choose 2} \,:\, v_{ij}v_{kl}\leq v_{ik}v_{jl}+v_{il}v_{jk}\ \text{\rm for }\{i,j\}\sqcup\{k,l\}\subseteq[n] \right\}.$$
Moreover, this containment is proper for $n\geq 8$.
\end{thm}

We prove this result in \rs{PMV-n-2}. In \rs{dim} we show that  both $\mathcal{MV}(n,2)$ and $\mathcal{PMV}(n,2)$ have full topological dimension.
Finally, we discuss applications of our results to intersection numbers of algebraic curves in $(\C\setminus\{0\})^2$ and
 tropical curves  in \rs{applications}.

\section{Pl\"ucker-type inequalities for four planar convex sets}\label{S:PTE}

In this section we prove that for any $K=(K_1,\dots, K_4)\in{\Km_2}^{\!4}$, the corresponding pure mixed area configuration
vector $$W_K=(V_{12}, V_{13}, V_{14}, V_{23}, V_{24},V_{34}),\quad \text{where }V_{ij}=\V(K_i,K_j),$$
satisfies the three Pl\"ucker-type inequalities \re{PTE-4}.
We start with the case when $K_1,\dots, K_4$ are centrally-symmetric. Although we will not need it to prove the general case, we include the proof
as it is short, transparent, and explains why we call the inequalities Pl\"ucker-type.

\subsection{The symmetric case} Let $\Km_2^{sym}$ denote the set of compact convex centrally-symmetric sets in the plane.
\begin{prop}\label{P:symmetric}
Let $K=(K_1,\dots, K_4)\in(\Km_2^{sym})^4$. Then the pure mixed area configuration vector
$W_K=(V_{12},V_{13},V_{14},V_{23},V_{24},V_{34})$ satisfies the Pl\"ucker-type inequalities \re{PTE-4}.
\end{prop}

\begin{pf} First suppose that $K_1,\dots, K_4$ are segments with direction vectors $v_1,v_2,v_3, v_4$. 
Since the Pl\"ucker-type inequalities are invariant under the permutations of the $K_i$, we may assume that
$v_1,\dots, v_4$ are ordered counterclockwise in such a way that $\det(v_i,v_j)\geq 0$ for 
$1\leq i<j\leq 4$. Note that $\V(K_i,K_j)=\frac{1}{2}\det(v_i,v_j)$ in this case (see \re{polarization-2}). Consider the $2\times 4$ matrix with columns $v_1,\dots, v_4$. Its maximal minors
$v_{ij}=\det(v_i,v_j)$ satisfy the Pl\"ucker relation (see \cite[Ch 14]{MillerSturmfels}):
\begin{equation}\label{e:PR}
v_{12}v_{34}-v_{13}v_{24}+v_{14}v_{23}=0.
\end{equation}
Thus, the middle inequality in \re{PTE-4} holds with equality. It remains to check the other two inequalities in \re{PTE-4}.
Indeed, by \re{PR} we have 
$$0\leq 2v_{12}v_{34}=2v_{12}v_{34}-\left(v_{12}v_{34}-v_{13}v_{24}+v_{14}v_{23}\right)=v_{12}v_{34}+v_{13}v_{24}-v_{14}v_{23}.$$
The other inequality is similar.

Now, if $K_1,\dots, K_4$ are centrally-symmetric convex polygons then each of them can be decomposed as the Minkowski sum of segments.
As the Pl\"ucker-type inequalities are linear in each $K_i$, the proof reduces to the case of segments considered above. Finally, 
the general case of centrally-symmetric sets follows by approximating them (in the Hausdorff metric) by convex polygons.
\end{pf}

\subsection{Special cases}
Now let $K_1,\dots, K_4$ be planar compact convex sets, not necessarily centrally-symmetric.
First we will look at a case when two of the $K_i$ are homothetic. 

\begin{prop}\label{P:homothetic}
Let $K=(K_1,\dots, K_4)\in{\Km_2\!}^4$ and assume that $K_3$  and $K_4$ are homothetic. Then the pure mixed area configuration vector
$W_K=(V_{12},V_{13},V_{14},V_{23},V_{24},V_{34})$ satisfies the Pl\"ucker-type inequalities.
\end{prop}

\begin{pf} By translating and rescaling  $K_4$ we may assume that $K_3=K_4$. Then the first two inequalities in \re{PTE-4}
are trivial and the third one states $2V_{13}V_{23}\geq V_{12}V_{33}$. This inequality is proved in \cite[Th 6.2]{SZ16}
and is an instance of a Bezout-type inequality for mixed volumes. See also
\cite[Lem 5.1]{BGL18} or \cite[Lem 4.1]{ABS20} for higher dimensional Bezout-type inequalities.
\end{pf}

Our next special case is when one of the $K_i$ is a segment.

\begin{prop}\label{P:segment}
Let $K=(K_1,\dots, K_4)\in{\Km_2\!}^4$ and assume that $K_4$ is a segment. Then the pure mixed area configuration vector
$W_K=(V_{12},V_{13},V_{14},V_{23},V_{24},V_{34})$ satisfies the Pl\"ucker-type inequalities.
\end{prop}

\begin{pf} By rescaling the sets in the tuple $K$ we may assume that $K_4$ is a unit segment and the $K_i$ have width $w_{u_4}(K_i)=2$ where $u_4$ is the direction orthogonal to $K_4$. Then $V_{i4}=1$ for $1\leq i\leq 3$ (see \rl{width}) and, up to permutations of the $K_i$, it is enough to prove 
\begin{equation}\label{e:triangle}
V_{12}\leq V_{13}+V_{23}. 
\end{equation}
Note that we may assume that $K_3$ is a segment, as replacing $K_3$ with a subset (and keeping the width equal to 2) only strengthens the inequality, by the monotonicity of the mixed volume. Let $w_{u_3}(K_i)=2a_i$, for $i=1,2$, where $u_3$ is the direction orthogonal to $K_3$. Similarly, for $i=1,2$, we may replace
$K_i$ with the parallelogram containing $K_i$ and having the same widths as $K_i$ in the directions $u_3$ and $u_4$, as this does not change $V_{i3}$ (by \rl{width}) and may only increase $V_{12}$. Now we may apply \rp{symmetric} or simply note that, by \rl{width}, $V_{i3}=a_i\lambda$ and $V_{12}=(a_1+a_2)\lambda$ where 
$\lambda$ is the length of $K_3$. In particular, \re{triangle} holds with equality in this case.
\end{pf}

\subsection{The case of triangles}\label{S:triangles}
In this subsection we assume that each $K_i$  is a triangle, which is the main case we need to consider.
We start with a special case of relative positions of the outer unit normals of the triangles.

\begin{lem}\label{L:share-vertex} 
Let $u_1,\dots, u_4$ be unit vectors (not necessarily distinct) contained in a half-plane and
ordered counterclockwise. Let $K_1,\dots, K_4$ be triangles in $\R^2$.
Suppose $u_1, u_4$ are outer unit normals of $K_i$ and $u_2,u_3$ are outer unit normals of $K_j$, for $1\leq i\neq j\leq 4$. Then
$$V_{ij}V_{kl}\leq V_{ik}V_{jl}+V_{il}V_{jk},$$
where $\{i,j\}\sqcup\{k,l\}=[4]$.
\end{lem} 

\begin{pf} After possible translations and dilations we may assume that $K_i$ is inscribed in $K_j$. By the assumption on the outer normals, $K_i$ and $K_j$
share a vertex. Then we can replace $K_i$ with a segment contained in $K_i$ and inscribed in $K_j$. In this case $V_{ik}$ and $V_{il}$ may only decrease and $V_{ij}$ does not change (see \rl{inscribed}). The lemma now follows from \rp{segment}.
\end{pf}

To handle the general case we use linear deformations of $K_1,\dots, K_4$ which reduce the problem to special configurations of the triangles.
Let $L$ be a triangle with vertices $A,B,C$ and outer normals $u_{AB}$, $u_{BC}$, and $u_{AC}$ to the corresponding sides of $L$.
We are interested in deformations that fix two of the normals, say $u_{AB}$ and $u_{AC}$, and change the third one, $u_{BC}$.
We achieve this by replacing the vertex $C$ with a vertex $C_t$ which depends on a parameter $t\in\R$ in such a way 
that $C_0=C$, $|{AC_t}|>|{AC}|$ for $t>0$, and $|{AC_t}|<|{AC}|$ for $t<0$, see \rf{deformation}. 
We also assume that the coordinates of $C_t$ are linear functions in $t$. We use $L(t)$ 
to denote such a deformation of $L$. Note that $L(t)$ is defined for all $t$ in an interval 
$[t_0,\infty)$, where $L(t_0)$ is a segment (when $B=C_{t_0}$) and $L(t)$ approaches
an infinite ``half-strip'', as $t\to\infty$.

\begin{figure}[h]
\begin{center}
\includegraphics[scale=.32]{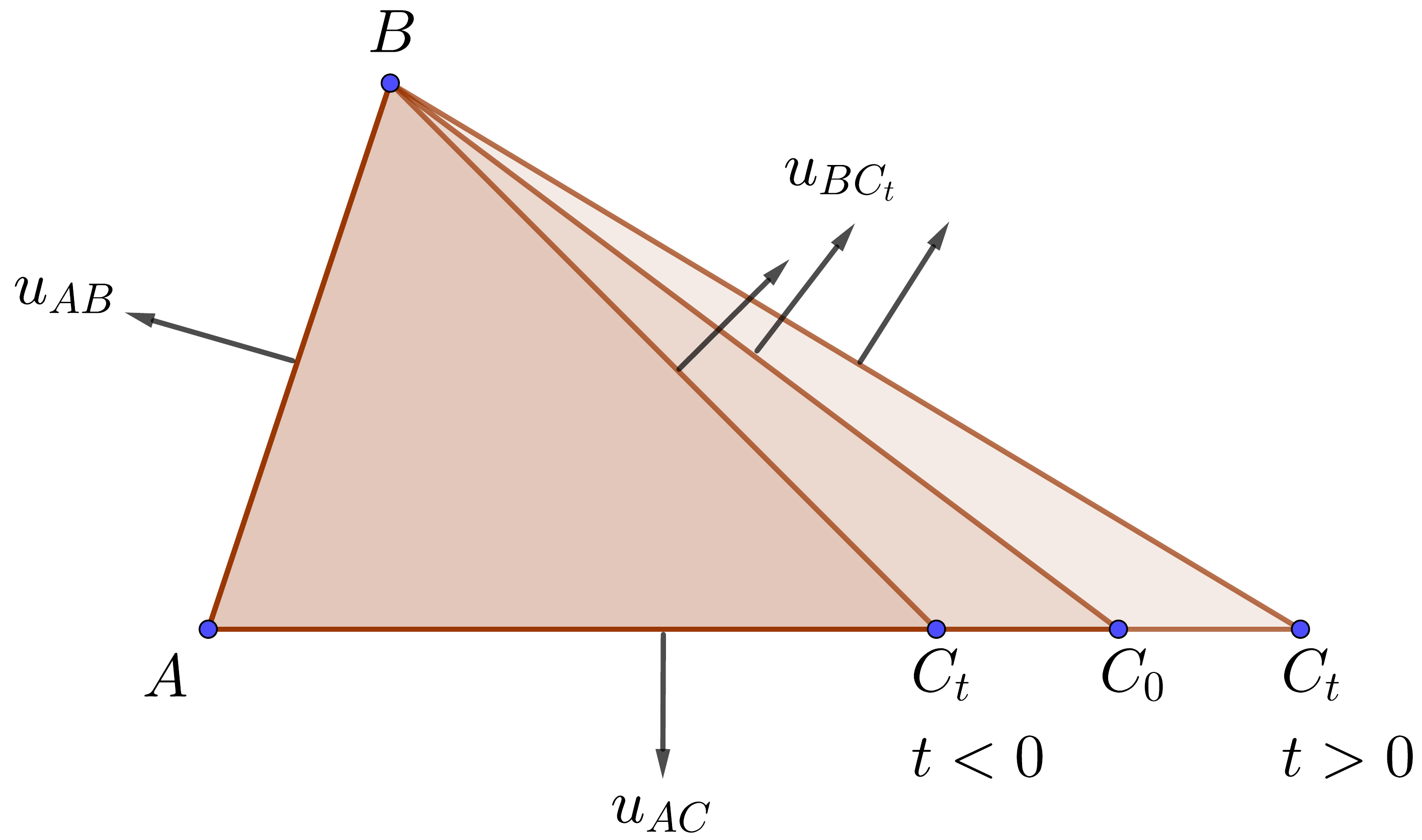}
\end{center}
\caption{A linear deformation $L(t)$ of a triangle $L=\triangle ABC$}
\label{F:deformation}
\end{figure}

\begin{lem}\label{L:lin-def} Let $L,M$ be triangles in $\R^2$ and 
let $L(t)$, $t\in[t_0,\infty)$ be a linear deformation of $L$, as above. 
Let $v_1,v_2$ be outer unit normals of $M$ and
assume that for all $t$ in a subset $T\subseteq[t_0,\infty)$
the normal $u_{BC_t}$ lies in the sector $[v_1,v_2]$. 
Then the mixed area $\V(L(t),M)$ is a linear function of $t$ on~$T$.
\end{lem}

\begin{pf}
By \re{mix-sum-2} for $K=M$ and $P=L(t)$ we have
\begin{equation}\label{e:Vij}\nonumber
2\V(L(t),M)=h_{M}(u_{AB})|AB|+h_{M}(u_{BC_t})|BC_t|+h_{M}(u_{AC_t})|AC_t|.
\end{equation}
The first term does not depend on $t$. For the second term, note that 
when $u_{BC_t}$ lies in $[v_1,v_2]$ there is a unique vertex
$x$ of $M$, independent of $t$, where the value $h_{M}(u_{BC_t})$ is attained. 
Then 
$$h_{M}(u_{BC_t})|BC_t|=\langle |BC_t|u_{BC_t},x\rangle.$$ 
As the coordinates of $|BC_t|u_{BC_t}$ are linear in the coordinates of $C_t$, we see that
$h_{M}(u_{BC_t})|BC_t|$ is a linear function in $t$.
The situation with the third term is similar and even simpler since $u_{AC_t}$ is constant.
\end{pf}

Now $K_1,\dots, K_4$ be triangles in $\R^2$. The outer unit normals to the sides of the $K_i$ define a configuration
of 12 unit vectors. We enumerate them counterclockwise $u_0,\dots, u_{11}$, where the indices
are defined mod 12. Pick a triangle $K_i$ and its normal $u_\ell$. Consider a deformation
$K_i(t)$ of $K_i$  with ``moving'' normal $u_\ell(t)$, for $t\in[t_0,\infty)$, as above. We have the following lemma.

\begin{lem}\label{L:half-strip} 
Let $K_1,\dots, K_4$ be triangles in $\R^2$ and let $K_i(t)$ be a deformation of $K_i$ as above.
Assume  $u_\ell(t)\in[u_{\ell-1},u_{\ell+1}]$ for all $t\in [t_1,\infty)$ for some $t_1\geq t_0$.
Then, if the function $-V_{12}V_{34}+V_{13}V_{24}+V_{14}V_{23}$ is non-increasing as $t\to\infty$, then it is constant
on $[t_1,\infty)$.
\end{lem}

\begin{pf} Let $F(t)=-V_{12}V_{34}+V_{13}V_{24}+V_{14}V_{23}$.
The only terms in $F(t)$ that depend on $t$ are the mixed areas $V_{ij}$ for $1\leq j\leq 4$, $j\neq i$.
By \rl{lin-def}, $F(t)$ is linear on $[t_1,\infty)$ and, hence,
we can write $F(t)=\lambda_1 t+\lambda_0$. Moreover, if $F(t)$
is non-increasing as $t\to\infty$, then $\lambda_1\leq 0$.
 If we rescale $K_i(t)$ by a factor $1/t$ we obtain the function $F(t)/t$. 
As $t\to\infty$, the triangle $\frac{1}{t}K_i(t)$ degenerates to a segment. Thus, $\lim_{t\to\infty}\frac{F(t)}{t}\geq 0$, by \rp{segment}. Therefore $\lambda_1\geq 0$, which shows that $F(t)$ is a constant function on $[t_1,\infty)$.
\end{pf}

We are now ready to prove the main result of this section.

\begin{prop}\label{P:triangles}
Let $K=(K_1,\dots, K_4)\in{\Km_2}^{\!4}$ be a tuple of triangles. Then the pure mixed area configuration vector
$W_K=(V_{12},V_{13},V_{14},V_{23},V_{24},V_{34})$ satisfied the Pl\"ucker-type inequalities.
\end{prop}

\begin{pf} 
As before, let $u_0,\dots, u_{11}$ be the outer unit normals to the sides of the triangles. 
Pick any $u_\ell$ which is strictly between
$u_{\ell-1}$ and $u_{\ell+1}$ and let $K_i$ be the triangle with this normal.
Consider a linear deformation $K_i(t)$ with the moving normal $u_\ell(t)$, as above.
We either increase $t$ in the positive direction or degrease $t$ in the negative direction in such a way that the function $F(t)=-V_{12}V_{34}+ V_{13}V_{24}+ V_{14}V_{23}$ is non-increasing, until either
$K_i(t)$ degenerates to a segment or $u_\ell(t)$ coincides with $u_{\ell-1}$ or $u_{\ell+1}$. 
Note that we can always avoid the case when $K_i(t)$ degenerates to an infinite half-strip, since
if $u_\ell(t)$ stays between $u_{\ell-1}$ and $u_{\ell+1}$ for all $t\to\infty$ then, by  \rl{half-strip},
$F(t)$ remains constant and we can choose the negative direction of $t$ instead.



Repeating the process we either arrive at the case when one of the $K_i$ is a segment, in which case the
statement follows from \rp{segment}, or all $K_i$ are triangles and 
in the configuration $\{u_0,\dots, u_{11}\}$ each normal is repeated at least once. 
Moreover, if a triple of normals in
$\{u_0,\dots, u_{11}\}$ belongs to the same triangle, then 
the normals positively span $\R^2$ (i.e. not all lie in a half-plane).

%

We next sort out all such configurations. We call the {\it multiplicity} of the normal in the configuration the number of times it is repeated.
First, assume that there is a normal with multiplicity~4. Then we are in the
case when \rl{share-vertex} applies to every pair $K_i$, $K_l$, $1\leq i\neq l\leq 4$, hence, all three Pl\"ucker-type inequalities hold.

Now we assume that each multiplicity is either 2 or 3. Suppose, for some $\ell$ the normals $u_{\ell-1}, u_\ell, u_{\ell+1}$ belong to the same triangle, say $K_4$. This means that each normal of the other three triangles is either contained in the sector $[u_{\ell+1},u_{\ell-1}]$ or coincides with $u_\ell$. As 
$\{u_{\ell+1},u_{\ell-1}\}$ does not positively span $\R^2$, exactly one of the normals of $K_1$, $K_2$, and $K_3$ must coincide with $u_\ell$.
In other words, $u_\ell$ has multiplicity 4. Therefore, no three consecutive normals may belong to the same triangle. This argument also shows that
a configuration with four normals, each with multiplicity 3 cannot exist.

Suppose there are six normals with multiplicity 2 each. To encode a possible configuration we use 
a $4\times 6$ matrix $M$ whose $(i,j)$-th entry is 1 if $j$-th normal belongs to $K_i$, and is 0, otherwise. We may assume that 
no two rows of $M$ are the same, as this corresponds to two triangles begin homothetic, which is already covered in \rp{homothetic}. Moreover, it is enough to consider classes of such matrixes under the action of $S_4$ on the rows of $M$ and of the dihedral group $D_6$ (as a subgroup of $S_6$) on the columns of $M$.
A computer search produces 5 such matrices. Figures~\ref{F:config-1}, \ref{F:config-2}, and \ref{F:config-3}  depict the configuration of four triangles and their six unit normals, corresponding to the matrices
$
\left(
\begin{matrix}
0 & 0 & 1 & 0 & 1 & 1\\
0 & 0 & 1 & 1 & 0 & 1\\
1 & 1 & 0 & 0 & 1 & 0\\
1 & 1 & 0 & 1 & 0 & 0
\end{matrix}
\right)
$, $\left(
\begin{matrix}
0 & 0 & 1 & 0 & 1 & 1\\
0 & 1 & 0 & 1 & 0 & 1\\
1 & 0 & 1 & 0 & 1 & 0\\
1 & 1 & 0 & 1 & 0 & 0
\end{matrix}
\right)$, and 
$\left(
\begin{matrix}
0 & 0 & 1 & 0 & 1 & 1\\
0 & 1 & 0 & 1 & 0 & 1\\
1 & 0 & 1 & 1 & 0 & 0\\
1 & 1 & 0 & 0 & 1 & 0
\end{matrix}
\right)$,
respectively. 
\begin{figure}[h]
\begin{center}
\includegraphics[scale=.26]{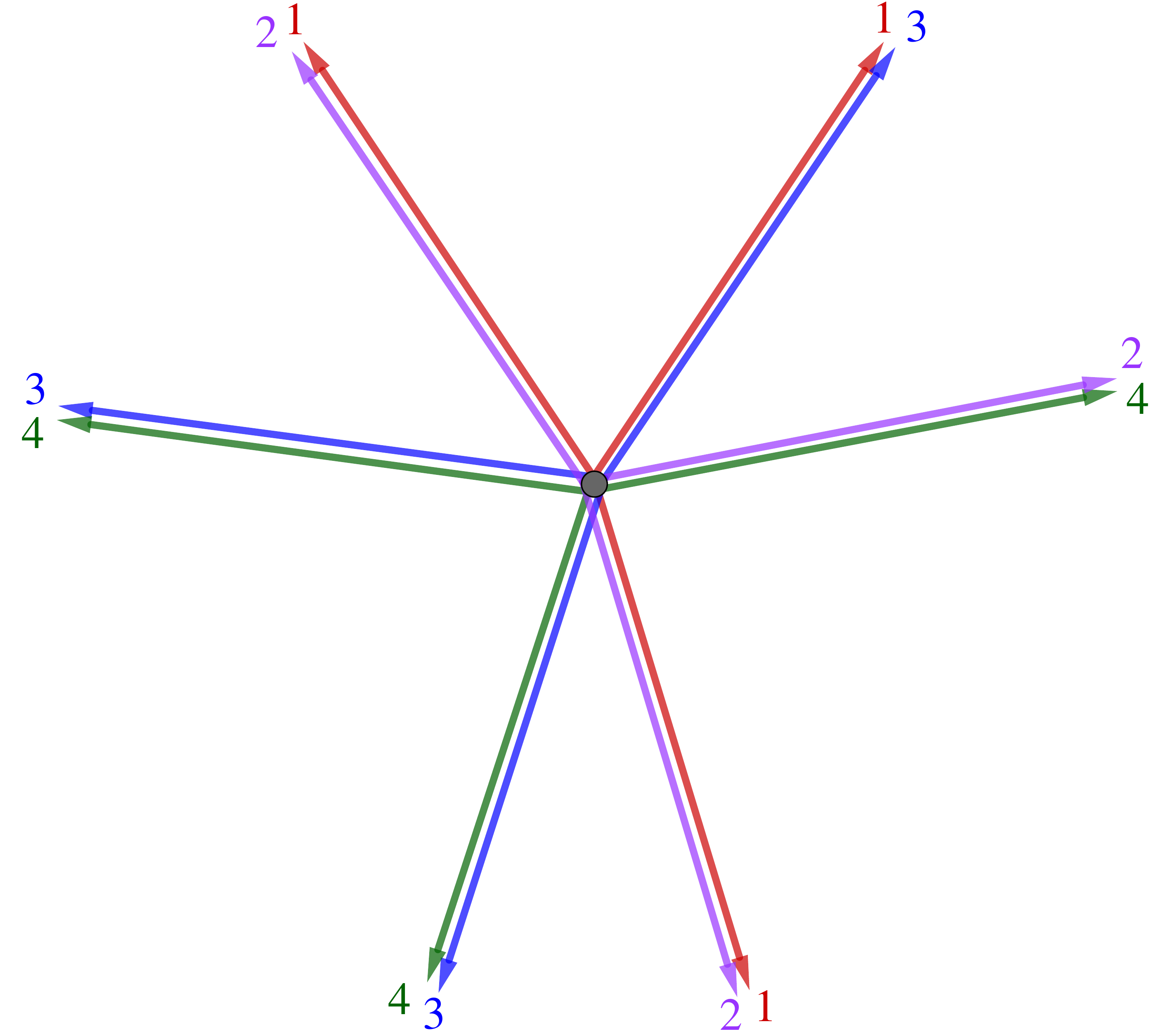} \ \ \ \ \ \ 
\includegraphics[scale=.26]{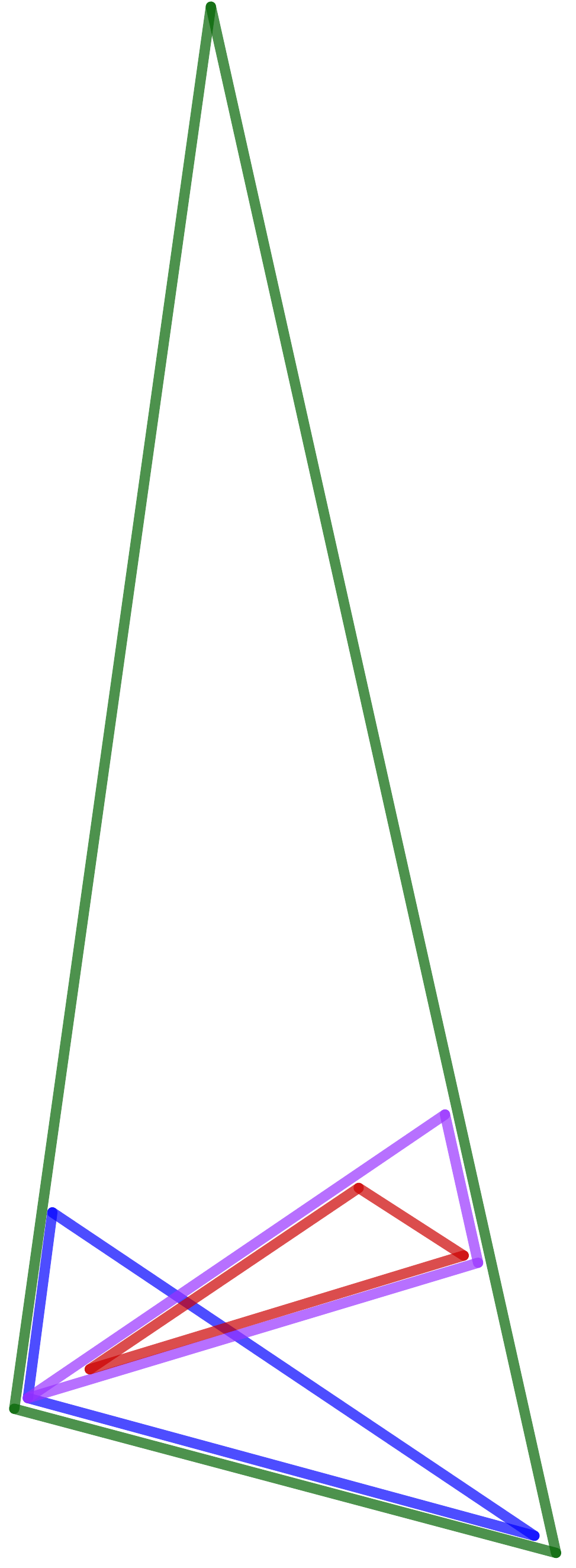}
\end{center}
\caption{First configuration of the unit normals and the triangles}
\label{F:config-1}
\end{figure}

\begin{figure}[h]
\begin{center}
\includegraphics[scale=.26]{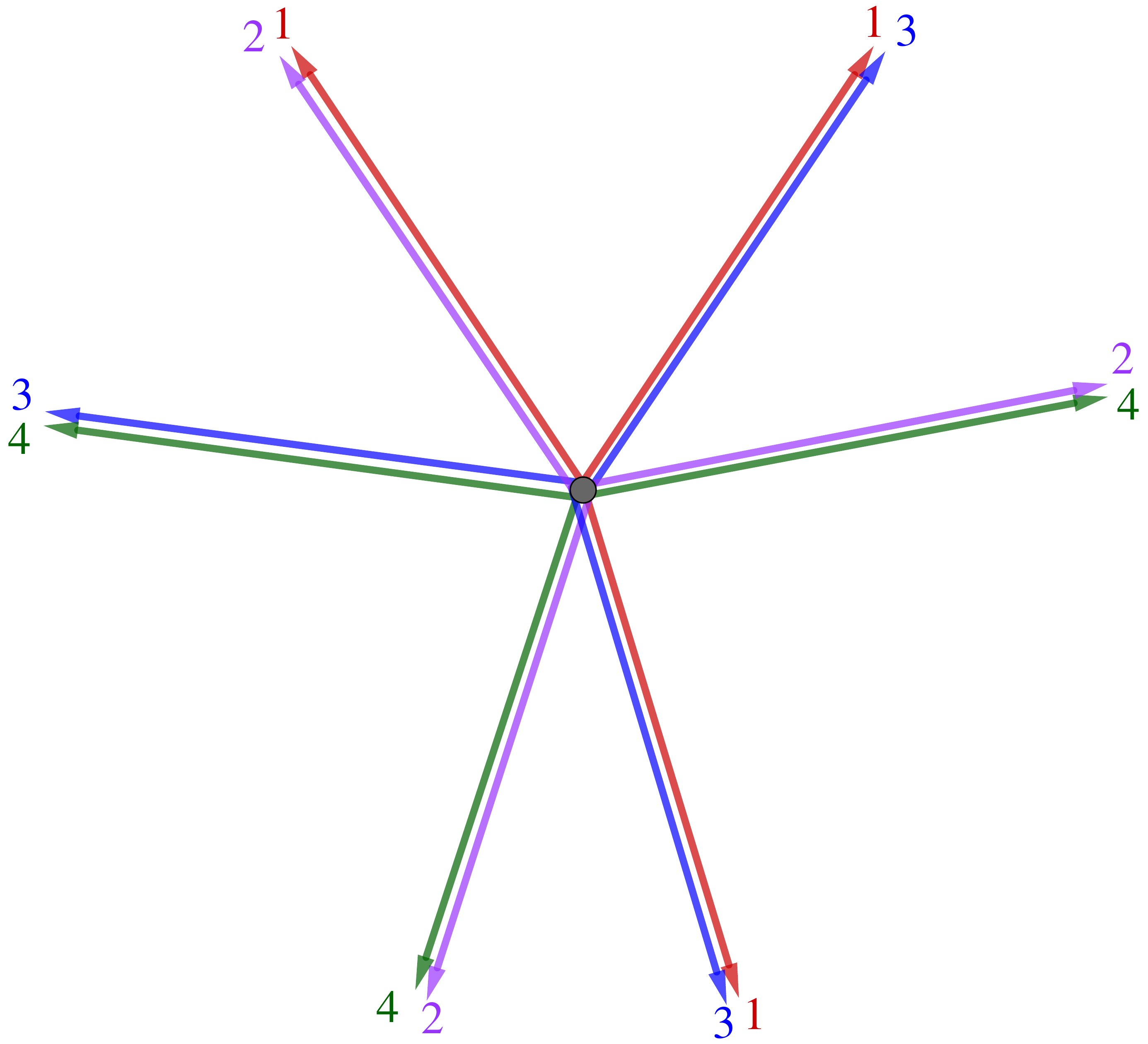} \ \ \ \ \ \ 
\includegraphics[scale=.26]{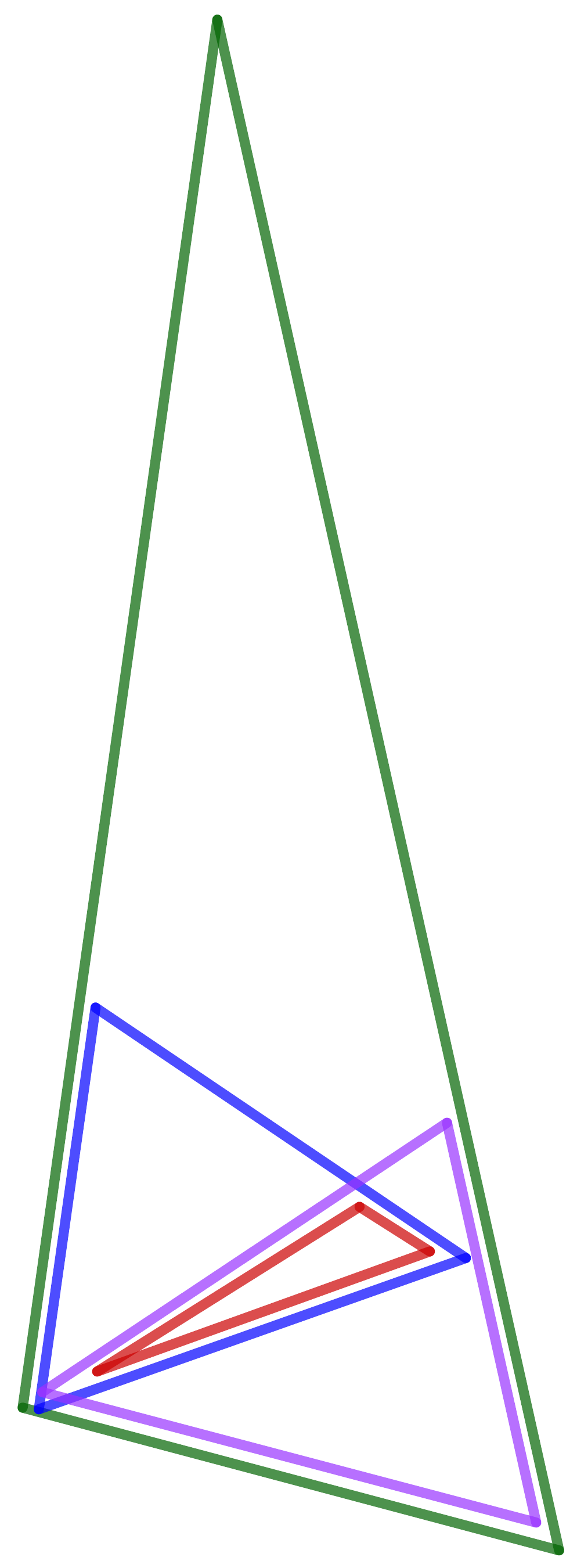}
\end{center}
\caption{Second configuration of the unit normals and the triangles}
\label{F:config-2}
\end{figure}

\begin{figure}[h]
\begin{center}
\includegraphics[scale=.26]{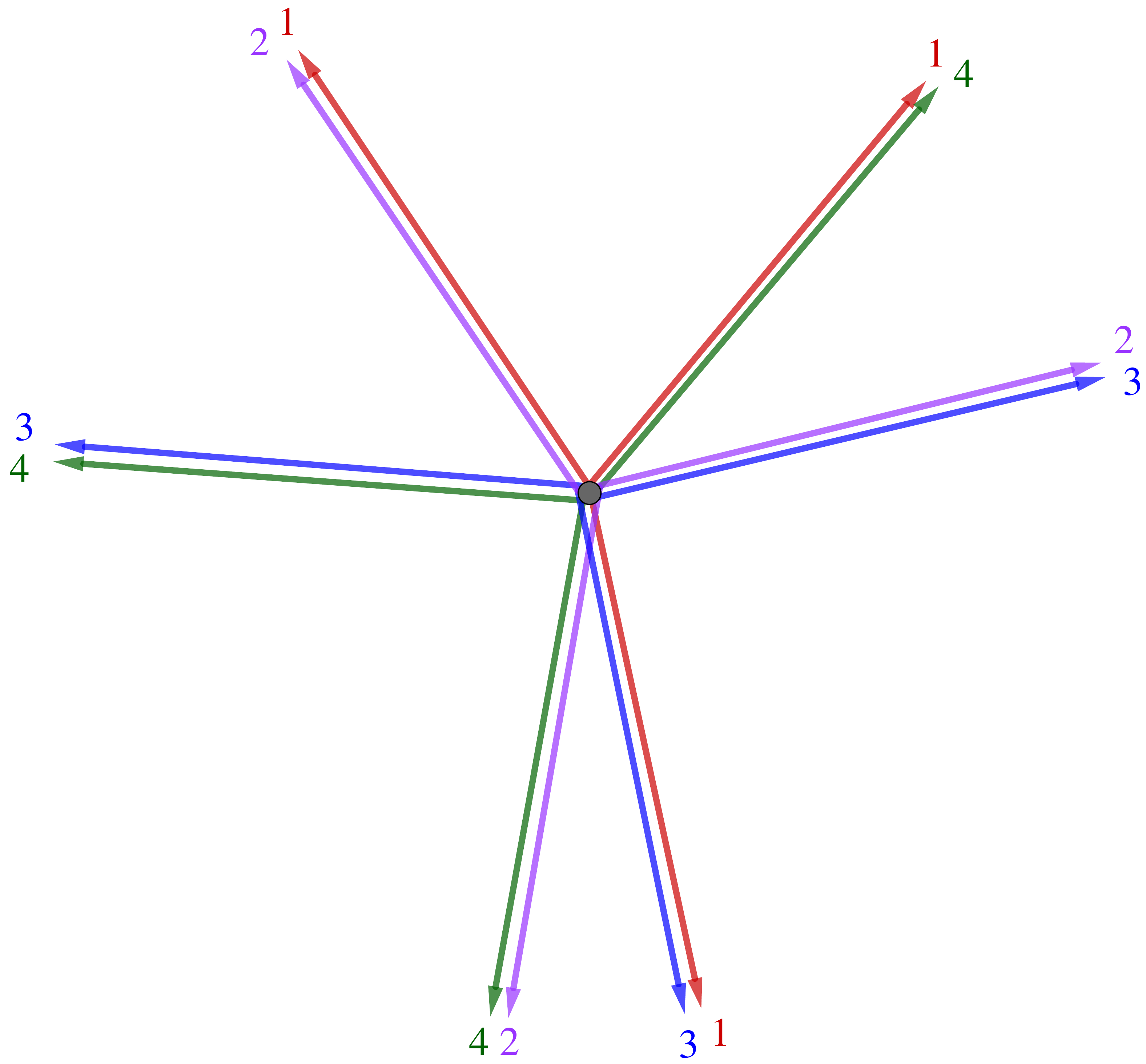} \ \ \ 
\includegraphics[scale=.21]{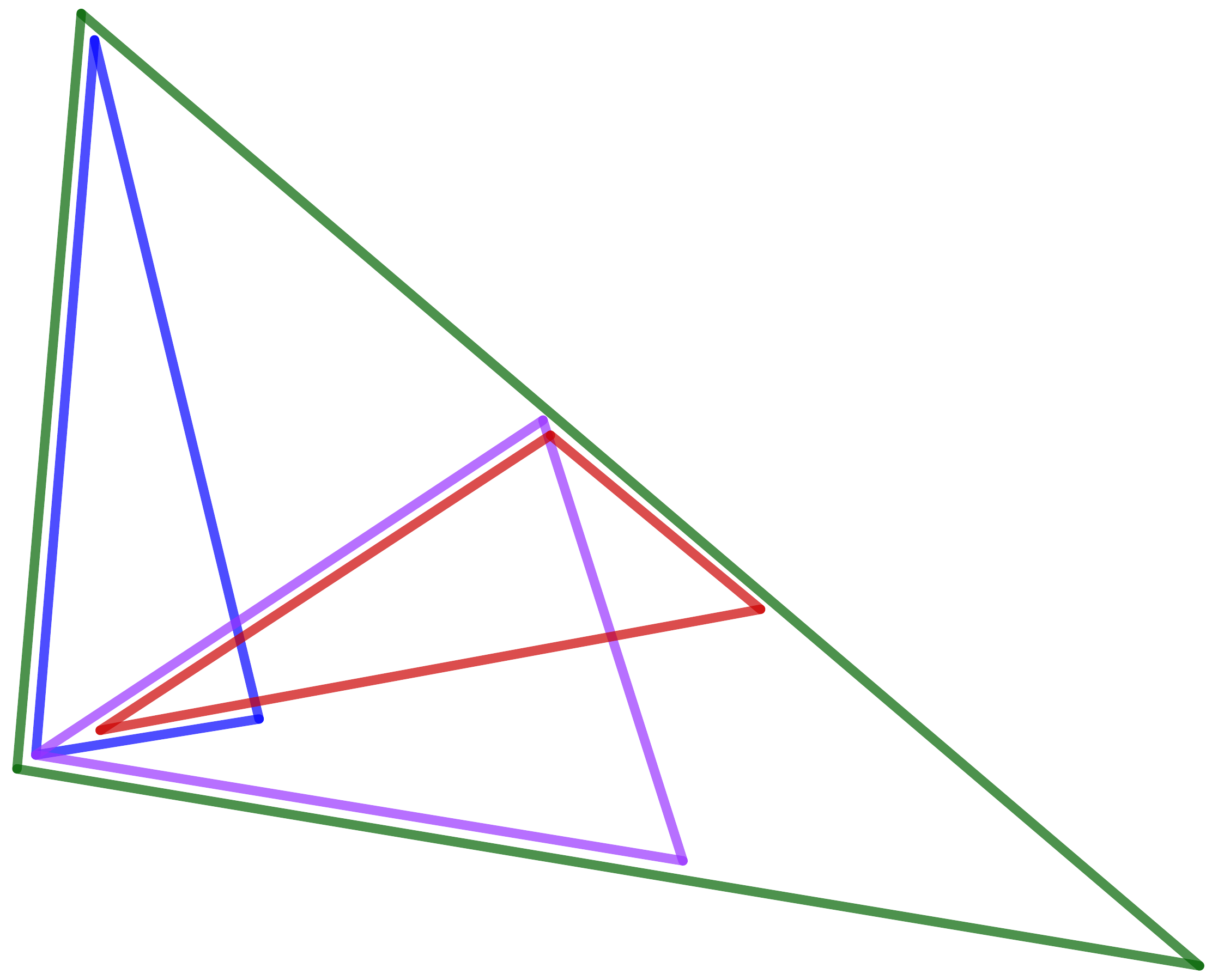}
\end{center}
\caption{Third configuration of the unit normals and the triangles}
\label{F:config-3}
\end{figure}
In the figures we rescaled and translated the triangles so that three of them are inscribed in the fourth. Note that
in each case, every pair of triangles $(K_i, K_j)$ shares a vertex (the left-most one). Thus, by
\rl{share-vertex}, all three Pl\"ucker-types inequalities are satisfied.  

The remaining two matrices
$\left(
\begin{matrix}
0 & 0 & 1 & 0 & 1 & 1\\
0 & 1 & 0 & 1 & 1 & 0\\
1 & 0 & 1 & 0 & 0 & 1\\
1 & 1 & 0 & 1 & 0 & 0
\end{matrix}
\right)$ and
$\left(
\begin{matrix}
0 & 0 & 1 & 0 & 1 & 1\\
0 & 1 & 1 & 0 & 0 & 1\\
1 & 0 & 0 & 1 & 1 & 0\\
1 & 1 & 0 & 1 & 0 & 0
\end{matrix}
\right)
$
do not correspond to any configuration of triangles. Indeed, for the latter matrix, note that 
all three normals of the fourth triangle are contained in the sector between two consecutive
normals of the third triangle, which is impossible. Similarly, for the former matrix: all normals
of the first triangle are contained between two consecutive
normals of the third one, a contradiction.


It remains to consider the case of two normals with multiplicity 3 and three with multiplicity 2. A computer search shows that,
up to the actions of $S_4$ and $D_5$, there is only one such configuration (see \rf{config-4}), corresponding to the matrix
$$
\left(
\begin{matrix}
0 & 1 & 0 & 1 & 1 \\
0 & 1 & 1 & 0 & 1 \\
1 & 0 & 1 & 0 & 1 \\
1 & 0 & 1 & 1 & 0
\end{matrix}
\right).
$$

\begin{figure}[h]
\begin{center}
\includegraphics[scale=.26]{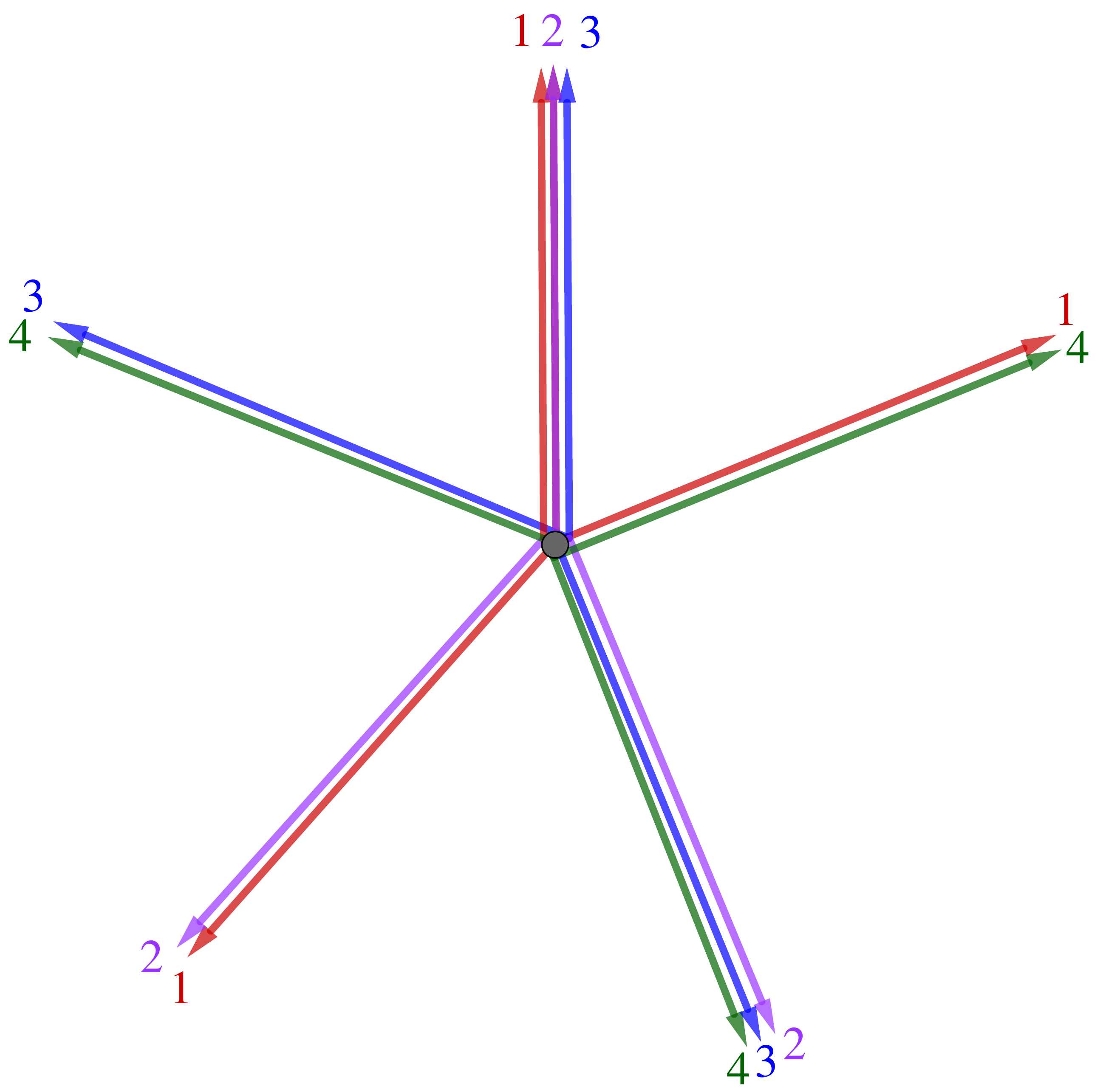} \ \ \ 
\includegraphics[scale=.26]{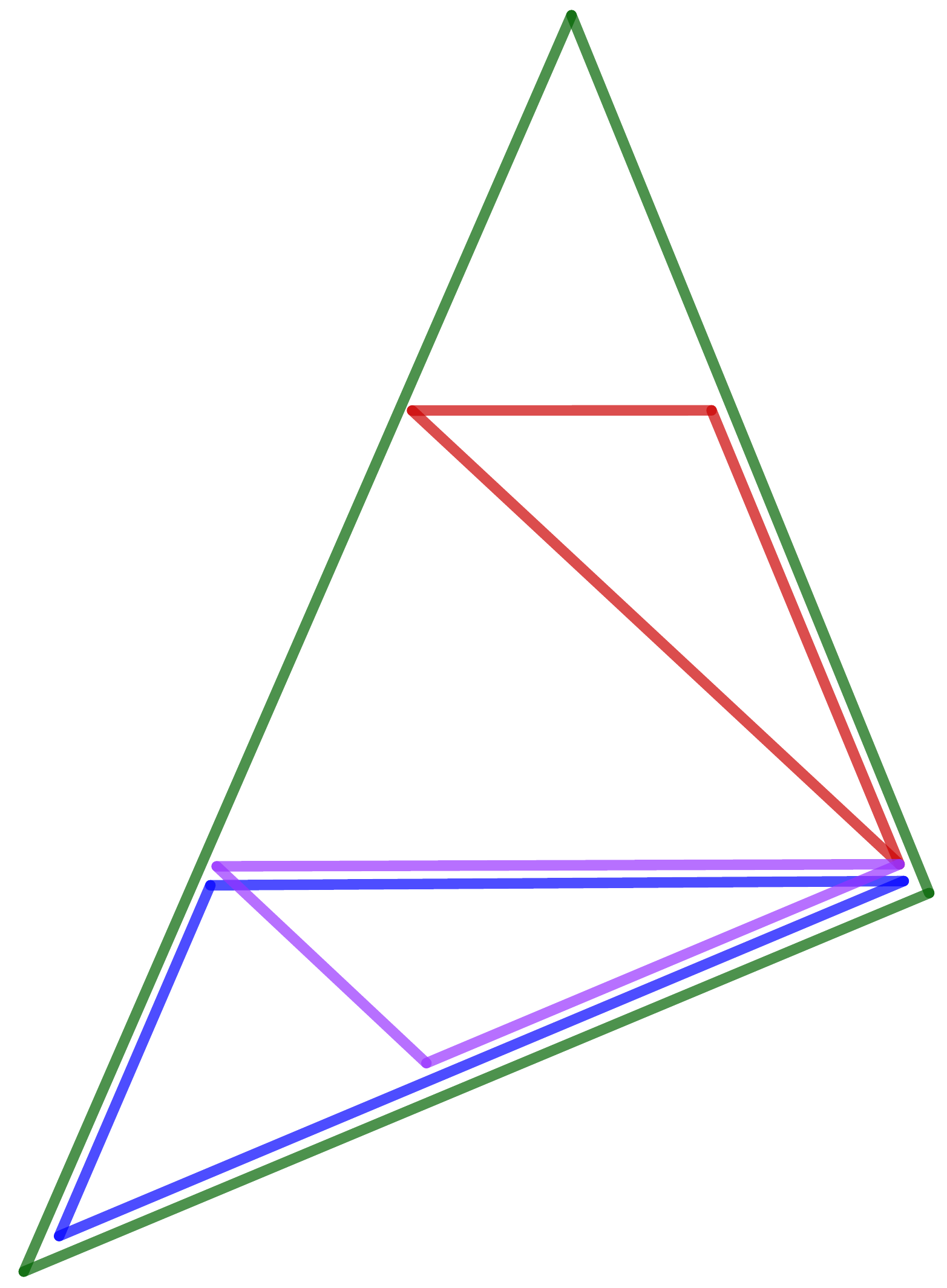}
\end{center}
\caption{The only configuration with five distinct unit normals}
\label{F:config-4}
\end{figure}

Again, we see that every pair of the triangles shares a vertex (the right-most one).
Applying  \rl{share-vertex} we obtain the three Pl\"ucker-type inequalities.
\end{pf}

\begin{rem} Although in the proof we used a computer to enumerate all possible configurations, this combinatorial problem is small enough to be solved by hand. The general study of combinatorial configurations with different kinds of balance conditions and symmetry is the subject of combinatorial design theory, see for example \cite{ColbournDinitz2007}.
\end{rem}

\section{Pl\"ucker space $\Pl_{n}$ and its graph representation}\label{S:graphs}

We next describe a combinatorial structure of the space defined by the Pl\"ucker-type inequalities which turns out to be useful when 
showing the reverse inclusion in \rt{PMV-4-2}. 
Denote
$$\Pl_{n} := \left\{v\in\R_{\geq 0}^{n\choose 2} \,:\, 
v_{ij}v_{kl}\leq v_{ik}v_{jl}+v_{il}v_{jk}\text{ for }\{i,j\}\sqcup\{k,l\}\subseteq[n] \right\}.$$
Note that $\Pl_{2}=\R_{\geq 0}$ and $\Pl_3=\R_{\geq 0}^3$. For convenience we set $\Pl_1:=\{0\}$.

\subsection{Graph representation}
To each $v=(v_{ij})\in\Pl_{n}$ we associate a weighted graph $G_v$ with vertex set $V$ and edge set $E$ as follows. We put $V=\{1,\dots, n\}$
and $\{i, j \} \in E$ if and only if  $v_{ij} >0$, in which case $v_{ij}$ is the weight of the edge $\{i,j\}$. We call $G_v$ the
{\it graph representing $v\in \Pl_{n} $}.

Throughout this section we use $\KG_n$ to denote the complete graph on $\{1,\dots,n\}$. Furthermore, 
$\KG_{n_1,\dots,n_s}$ will denote a complete multipartite graph on $\{1,\dots, m\}$ corresponding to a partition 
$m=n_1 + \cdots + n_s $ with $s \ge 2$ and $1\leq n_1\leq \dots\leq n_s$.

The following proposition shows that the graphs $G_v$ have a special combinatorial property.

\begin{prop} \label{P:no:isolated:points:in:G}
For any $v\in\Pl_n$ the graph $G_v$ is a union of a complete multipartite graph and some isolated nodes. 
\end{prop} 
\begin{proof} Without loss of generality we may assume that $\{m+1,\dots, n\}$ are the isolated nodes of $G_v$ for some $0\leq m\leq n$. Let $G$
be the induced subgraph of $G_v$ on $\{1,\dots, m\}$ which has no isolated nodes. We need to show that $G$ is a complete multipartite graph.
Consider the  Pl\"ucker-type inequality $v_{ij}v_{kl}\leq v_{ik}v_{jl}+v_{il}v_{jk}${ for some }$\{i,j,k,l\}$ satisfying $\{i,j\}\sqcup\{k,l\}\subseteq[m]$. 
Note that if $v_{ik}=v_{il}=0$ then either $v_{ij}=0$ or $v_{kl}=0$ (or both). Combinatorially, this means that $G$ cannot contain
the following graphs on four nodes as induced subgraphs: (a) the union of $\KG_3$ and $\KG_2$ sharing exactly one node; (b) the path of length three; (c) the disjoint union of two copies of $\KG_2$, see \rf{three forbidden graphs}.

\begin{figure}[h]
\begin{center}
\includegraphics[scale=.32]{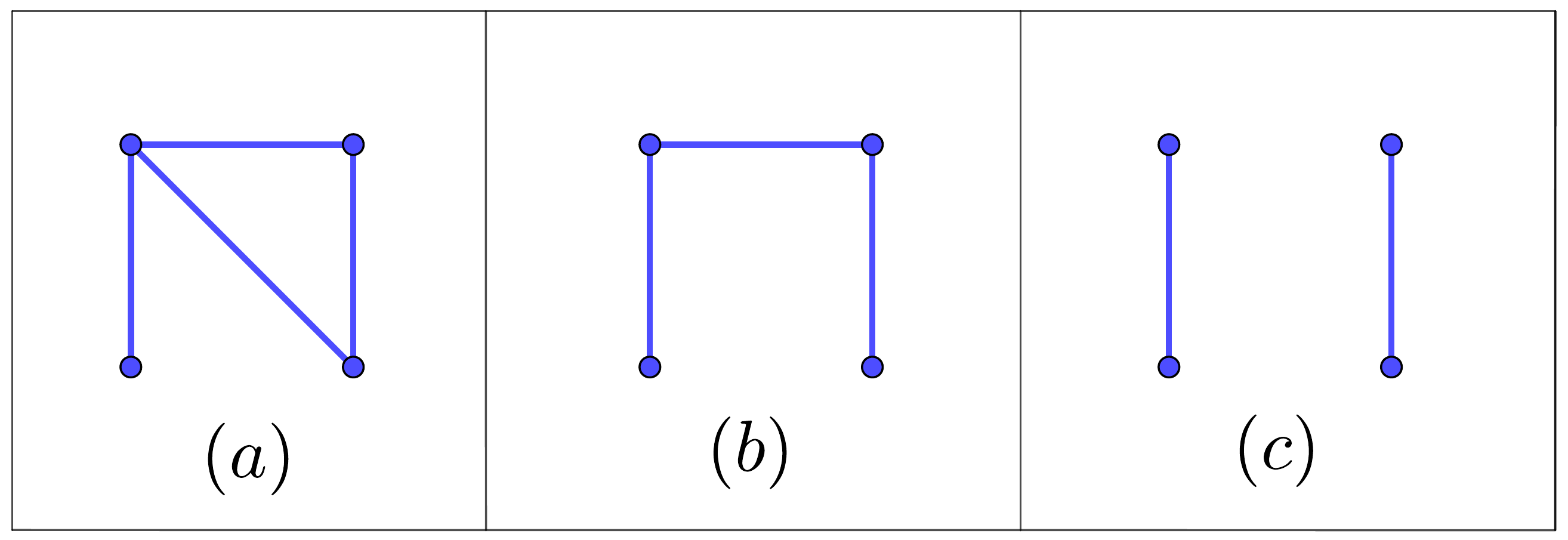}
\end{center}
\caption{Three forbidden graphs}
\label{F:three forbidden graphs}
\end{figure}


To check that $G$ is a complete multipartite graph, we need to show that each connected component of the complement of $G$ is a complete graph.
So let $a,b,c$ be three distinct nodes such that both $\{a,b\}$ and $\{b,c\}$ are edges of the complement $\overline{G}$. It suffices to show that $\{a,c\}$ is also an edge of $\overline{G}$.
	
	Since neither $a$ nor $b$ is an isolated node of $G$, each of them has a neighbor in $G$, which we denote by $a'$ and $b'$, respectively. If $a' \ne b'$, then consider the subgraph of $G$ induced by $\{a,a',b,b'\}$. Then $\{a',b\}$ must be an edge of $G$, otherwise $G$ would contain one of
the graphs (a)--(c) as an induced subgraph. We, thus, have shown that $a$ and $b$ must share a neighbor in $G$. 
Using the same argument for $c$ in place of $b$, we conclude that $\{a',c\}$ must also be an edge of $G$. Summarizing, we conclude that 
$\{a',a\}$, $\{a',b\}$, and $\{a',c\}$ are edges of $G$, whereas $\{a, b\}$ and $\{b,c\}$ are not edges of $G$. 
It follows that $\{a,c\}$ is not an edge of $G$, as otherwise the subgraph induced by $\{a, a', b, c\}$ would be the graph (a), which is impossible. 
\end{proof} 
%
%

We will need the following simple observation.

\begin{lem}\label{L:square}  Let $v \in \Pl_{n}$ such that $v_{ik}=0$ for some $\{i,k\}\subset[n]$. Then for any $\{j,l\}$ satisfying
$\{i,j\}\sqcup\{k,l\}\subseteq[n]$ we have $v_{ij}v_{kl}=v_{il}v_{jk}$.
\end{lem}
\begin{proof} This follows immediately from the two Pl\"ucker-type inequalities $v_{ij}v_{kl}\leq v_{ik}v_{jl}+v_{il}v_{jk}$
and $v_{il}v_{jk} \leq v_{ij}v_{kl}+v_{ik}v_{jl}$.
\end{proof}

There is a natural embedding of $\Pl_{n-1}$ as a subset of $\Pl_n$ via
\begin{equation}\label{e:embedding}
\iota_n:\Pl_{n-1}\to\Pl_{n},\quad v\mapsto (v_{ij}),\ \text{where}\ v_{in}=0\ \text{for all}\ i\in [n-1].
\end{equation}
Indeed, if $v_{in}=0$ for all $i\in [n-1]$ then all the Pl\"ucker-type inequalities of the form
$$v_{ij}v_{nl}\leq v_{in}v_{jl}+v_{il}v_{jn}$$ become trivial and only the defining inequalities for $\Pl_{n-1}$ remain.
On the combinatorial side, $G_{\iota_n(v)}$ is the graph obtained from $G_v$ by appending the $n$-th node as an isolated node.

\subsection{Action of $\R_{>0}^n\times\Sym_n$} 
We next consider the action of the group $\G_n:=\R_{>0}^n\times\Sym_n$ on $\R^{n\choose 2}$ given by
$$(\lambda,\sig)\cdot (v_{ij})=\left(\lambda_{\sig(i)}\lambda_{\sig(j)} v_{\sig(i)\sig(j)}\right),\quad\text{for } (v_{ij})\in\R^{n\choose 2}.$$ 
Here $\Sym_n$ is the symmetric group on $n$ elements. It is easy to see that both
$\PMV(n,2)$ and $\Pl_{n}$ are invariant under the action of $\G_n$. For $\PMV(n,2)$ it follows since we can permute and rescale by positive
scalars the convex sets in the tuple $K=(K_1,\dots, K_n)$. For $\Pl_{n}$ this is true as the  Pl\"ucker-type inequalities \re{PTE}
are symmetric and homogeneous. Therefore, when proving $\Pl_{n}\subseteq\PMV(n,2)$ it is enough to check that
a representative from each $\G_n$-orbit in $\Pl_{n}$ lies in $\PMV(n,2)$. 

Let $\tilde \Pl_n=\Pl_n/\G_n$ be the quotient space. Using the embedding \re{embedding} we can identify $\tilde \Pl_{n-1}$ with a subset of $\tilde\Pl_n$. We have the following ascending chain 
$$\tilde \Pl_1\subset \tilde \Pl_2\subset\cdots\subset \tilde \Pl_{n-1}\subset \tilde \Pl_n\subset\cdots$$
The next proposition describes the difference set $\tilde \Pl_n\setminus \tilde \Pl_{n-1}$.

\begin{prop}\label{P:dim-of-orbits}
For a partition $n=n_1 + \cdots + n_s $ with $s \ge 2$ and $1\leq n_1\leq \dots\leq n_s$, let 
$$\Pl_{n_1,\dots,n_s}:=\{v\in\Pl_{n}\,:\,G_v\cong\KG_{n_1,\dots,n_s}\}.$$
Then 
$$\tilde\Pl_{n}\setminus\tilde\Pl_{n-1}=\bigcup_{n=n_1 + \cdots + n_s}\tilde\Pl_{n_1,\dots,n_s},$$
where the union is over all partitions of $n$ as above. Moreover, $\dim \tilde\Pl_{n_1,n_2}=0$ and 
$$
\dim \tilde\Pl_{n_1,\dots,n_s}\leq \sum_{1\leq i<j\leq s}n_in_j-n.
$$ 
\end{prop}
\begin{proof} The first statement follows from \rp{no:isolated:points:in:G}. For the second statement, first note that 
$\dim\Pl_{n_1,\dots,n_s}$ is at most the number of edges in $G:=\KG_{n_1,\dots,n_s}$, which is $\sum_{1\leq i<j\leq s}n_in_j$. Let $s>2$.
We claim that for any $v\in\Pl_{n_1,\dots,n_s}$ there exists $g\in\G_n$ such that  $n$ of the coordinates of $g\cdot v$ are equal to $1$. This implies the bound for $\dim\tilde\Pl_{n_1,\dots,n_s}$. 

To prove the
claim, we show that there exist a set $E'$ of $n$ edges of $G$ and a vector $\lambda\in\R_{>0}^n$ such that $\lambda_i\lambda_jv_{ij}=1$ for all $\{i,j\}\in E'$. Let $M$ be the incidence matrix of $G$ whose rows correspond to the edges of $G$. It is well-known that $\rk(M)=n$ when $G$ contains an odd cycle (which happens if and only if $s>2$) and $\rk(M)=n-1$, otherwise. (See, for example, \cite[Prop. 1.3]{ OhsugiHibi97}.) 
Let $E'$ be a subset of edges of $G$ for which the corresponding $n\times n$ 
submatrix $M'$ is of full rank. We obtain a polynomial system in $\lambda_i$, $1\leq i\leq n$
$$\{\lambda_i\lambda_jv_{ij}=1\,:\,\{i,j\}\in E'\}.$$
By taking the natural logarithm of both sides we can reduce it to an $n\times n$ 
linear system in $t_i=\ln\lambda_i$, $1\leq i\leq n$
$$\{t_i+t_j=-\ln(v_{ij})\,:\,\{i,j\}\in E'\}.$$
Since $\rk(M')=n$,  this system has a solution for any values of $v_{ij}\in\R_{>0}$, $\{i,j\}\in E'$, which proves the claim.


Now let $s=2$, i.e. $G$ is bipartite. Assume $V_1:=\{1,\dots, n_1\}$ and $V_2:=\{n_1+1,\dots, n\}$ are the two parts of the nodes of $G$. Consider
the set of $n-1$ edges 
$$E'=\left\{\{1,j\}, \{i,n_1+1\}\,:\,i\in V_1, j\in V_2 \right\},$$ 
and let $M'$ be the corresponding $(n-1)\times n$ submatrix of the incidence matrix of $G$. Since $E'$ is a tree, $M'$ is of full rank. As above, 
this implies that, after the action by an element of $\G_n$, the coordinates of $v$ satisfy $v_{1,j}=v_{i,n_1+1}=1$ for all  $i\in V_1$, $j\in V_2$.
By \rl{square}, we have $v_{i,j}v_{1,n_1+1}=v_{i,n_1+1}v_{1,j}$ which implies $v_{i,j}=1$ for all $i\in V_1$, $j\in V_2$. In other words, 
$\dim \tilde\Pl_{n_1,n_2}=0$. 
\end{proof}

This proposition shows that one can view $\tilde\Pl_n$ as a union of non-labeled weighted graphs on $n$ nodes, where each graph is a complete multipartite graph on $m$ nodes together with $n-m$ isolated nodes, for some $0\leq m\leq n$. Moreover, if the graph is non-bipartite, $m$ of the weights can be chosen to be 1; and if it is bipartite all of the weights can be chosen to be 1. Note, that the bound on the dimension can be strict. For example, consider
$\tilde\Pl_{1,1,2}$. Using the group action we can make four of the weights to be $1$ (as long as the corresponding edges do not form a 4-cycle), e.g.,
we can put $v_{12}=v_{23}=v_{13}=v_{24}=1$. (Here we chose the partition $\{2\}\sqcup\{3\}\sqcup\{1,4\}=[4]$.) Since $v_{14}=0$, \rl{square} implies that $v_{34}=1$ as well. Therefore, $\dim\tilde\Pl_{1,1,2}=0$.
Applying \rp{dim-of-orbits} for $n=4$ we see that $\tilde\Pl_4$ consists of seven 0-dimensional and one 2-dimensional components corresponding  to the eight weighted graphs depicted in \rf{P_4}.

\begin{figure}[h]
\begin{center}
\includegraphics[scale=.32]{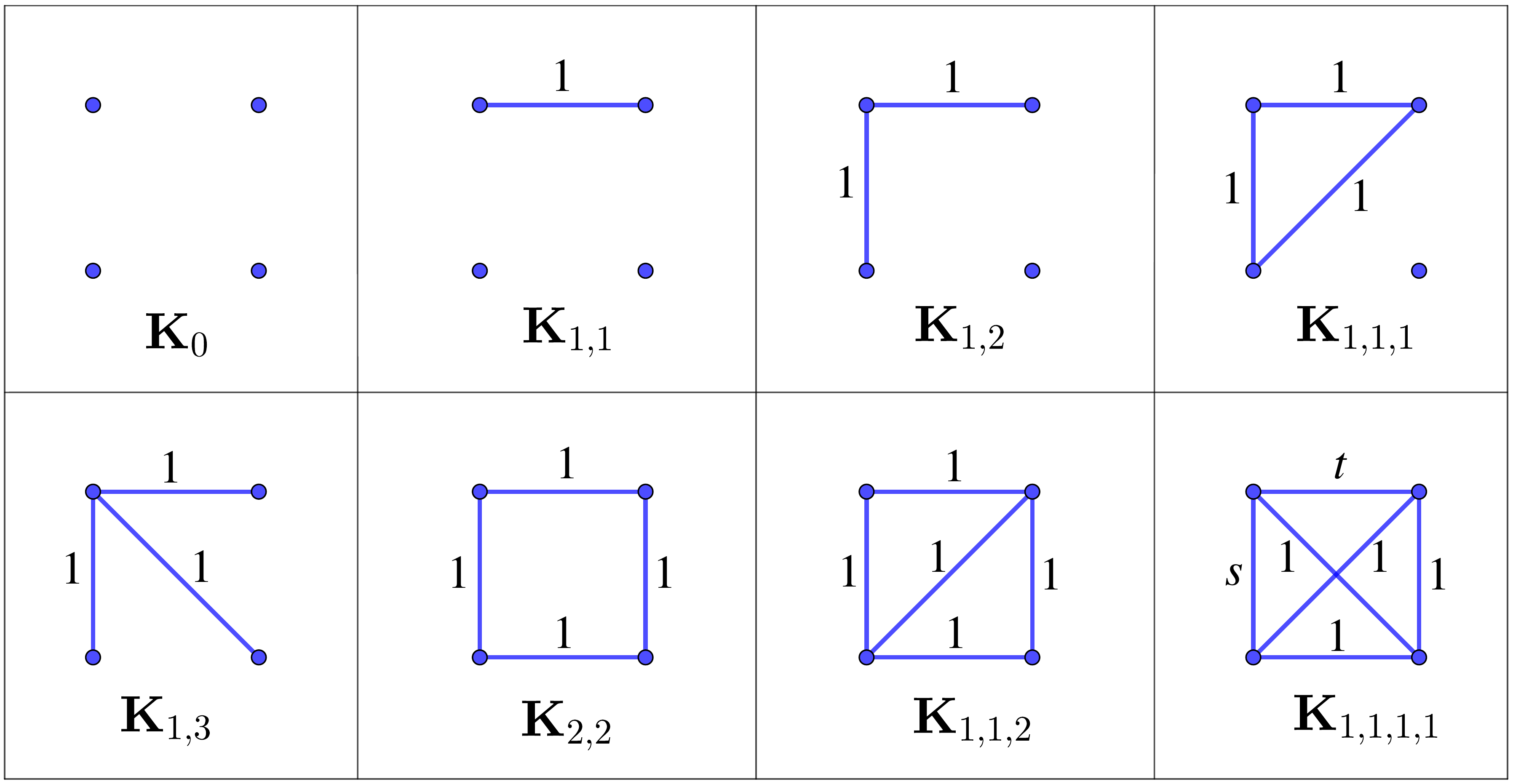}
\end{center}
\caption{The eight weighted graphs corresponding to the eight components of $\tilde\Pl_4$.}
\label{F:P_4}
\end{figure}

\section{The proof of $\PMV(4,2)=\Pl_4$}

In this section we combine the results of Sections \ref{S:PTE} and \ref{S:graphs} to show the equality $\PMV(4,2)=\Pl_4$. 

\begin{proof}[Proof of \rt{PMV-4-2}] Consider $K=(K_1,\dots, K_4)\in{\Km_2\!}^4$ and let $W_K=(V_{12},V_{13},V_{14},V_{23},V_{24},V_{34})$
be the corresponding pure mixed area configuration vector. We need to show that $W_K$ lies in $\Pl_4$, i.e. 
satisfies the Pl\"ucker-type inequalities.
As compact convex sets are approximated by convex polygons in the Hausdorff metric, we may assume that each $K_i$ is a convex polygon. Furthermore, each planar convex polygon is the Minkowski sum of segments and
triangles. Since the Pl\"ucker-type inequalities are linear in the $K_i$, we may assume that each $K_i$ is either a segment or a triangle.
If at least one of the $K_i$ is a segment, by \rp{segment}, $W_K$ satisfies the Pl\"ucker-type inequalities. If all $K_i$ are triangles, we apply 
\rp{triangles}. Therefore, $\PMV(4,2)\subseteq\Pl_4$.

According to the above discussion, to show $\Pl_4\subseteq\PMV(4,2)$ it is enough to check that a representative $v$ from each
of the eight components of $\tilde\Pl_4$ can be realized as a point in $\PMV(4,2)$, see  \rf{P_4}. We make the
choices for the four convex sets $K_1,\dots, K_4$ according to Table~\ref{table-P4}. There we use $\{0\}$ for the origin and 
$I_1,I_2,I_3$ to denote the 
segments $[0,\sqrt{2}e_1]$, $[0,\sqrt{2}e_2]$, and $[0,\sqrt{2}(e_1+e_2)]$, respectively. Note
that $\V(I_i,I_i)=0$ and $\V(I_i,I_j)=1$ for $i\neq j$.
\begin{center}
  \begin{table}[h]
  \renewcommand{\arraystretch}{1.3}
  \begin{tabular}{|l|c|}
    \hline
    $G_v$ & choice of $K_1,\dots, K_4$ \\ 
    \hline
    \hline
    $\KG_0$ & $K_1=K_2=K_3=K_4=\{0\}$ \\ 
    \hline
    $\KG_{1,1}$ & $K_1=I_1$,\, $K_2=I_2$,\, $K_3=K_4=\{0\}$ \\ 
    \hline
    $\KG_{1,2}$ & $K_1=I_1$,\, $K_2=K_3=I_3$,\, $K_4=\{0\}$\\ 
     \hline
    $\KG_{1,1,1}$ & $K_1=I_1$,\, $K_2=I_2$,\, $K_3=I_3$,\, $K_4=\{0\}$\\ 
     \hline
    $\KG_{1,3}$ & $K_1=I_1$,\, $K_2=K_3=K_4=I_2$\\ 
     \hline
    $\KG_{2,2}$ & $K_1=K_4=I_1$,\, $K_2=K_3=I_2$\\ 
     \hline
    $\KG_{1,1,2}$ & $K_1=K_4=I_1$,\, $K_2=I_2$, $K_3=I_3$\\ 
     \hline
    $\KG_{1,1,1,1}$ & \ \ $K_1=\frac{1}{2}(s-t+1)I_1+\frac{1}{2}(-s+t+1)I_2+\frac{1}{2}(s+t-1)I_3$,\ \ \\
    & $K_2=I_1$,\, $K_3=I_2$,\, $K_4=I_3$\\
    \hline
    \end{tabular}
 \vspace{.2cm}
\caption{Realizing representatives in $\tilde\Pl_4$ as elements of $\PMV(4,2)$.}
\label{table-P4}
  \end{table}
\end{center}

We remark that in all of the cases except the last one we realize $v$ using (possibly degenerate) segments. For 
$G_v\cong \KG_{1,1,1,1}$, one of the $K_i$ is a hexagon and the others are segments. The lengths of the hexagon are non-negative
as guaranteed by the Pl\"ucker-type inequalities: $-s+t+1\geq 0$, $s-t+1\geq 0$, $s+t-1\geq 0$.
\end{proof}



\section{Semialgebraicity of $\overline{\PMV(n,2)}$}\label{S:PMV-n-2}

In this section we study semialgebraicity of the pure mixed volume configuration space. 
Recall that a subset $S\subset\R^k$ is called a  {\it basic closed semialgebraic set} if there exist polynomials
$f_1,\dots, f_m$ in the polynomial ring $\R[x_1,\dots,x_k]$ such that $S=\{x\in\R^k\,:\,f_i(x)\geq 0, i\in[m]\}$. A {\it semialgebraic set} is a 
set generated by a finite sequence of set-theoretic operations (union, intersection, and complement) on
basic closed semialgebraic sets. We refer to \cite{BCR} for all necessary definitions regarding semialgebraicity.

As we have seen in the previous section, $\PMV(n,2)=\Pl_n$ for $n=2,3,4$. In particular, it is a basic closed semialgebraic set. 
For the general case, the results of
\rs{PTE} imply that $\PMV(n,2)\subseteq \Pl_n$ for any $n\geq 4$. In \rt{PMV-8-2} below we show that
this inclusion is proper starting with $n\geq 8$. We do not know whether $\PMV(n,2)= \Pl_n$ for $n=5,6,7$. Neither do we know if 
$\PMV(n,2)$ is closed for $n\geq 5$. Nevertheless, we can show that the closure $\overline{\PMV(n,2)}$ is a semialgebraic subset of $\R^{n\choose 2}$, for all $n\geq 2$, see \rt{semialgebraicity}.

We start with a slightly more general situation. We say a map $\phi:\Km_2\to\R^k$
is {\it additive} if
$$\phi(\alpha K+\beta L)=\alpha\,\phi(K)+\beta\,\phi(L),\quad\text{for all } K,L\in\Km_2 \text{ and } \alpha,\beta\in\R_{\geq 0}.$$

Note that the image $\phi(\Km_2)$ is a cone in $\R^k$. We say that $\phi$ is {\it continuous} if it is continuous with respect to the Hausdorff metric
in $\Km_2$ and Euclidean metric in $\R^k$. 
Let $\Pm_2$ denote the space of all convex polytopes in $\R^2$, and $\Pm_{2,m}$ the subset of convex polytopes with at most $m$ vertices. We have the following lemma.

\begin{prop}\label{P:cone-generators} Let $\phi:\Km_2\to\R^k$ be additive. Then  $\phi(\Pm_2)=\phi(\Pm_{2,3k})$. Consequently,
when $\phi$ is additive and continuous, the closure $\overline{\phi(\Km_2)}$ in $\R^k$ coincides with $\overline{\phi(\Pm_{2,3k})}$. 
\end{prop}

\begin{proof} As we mentioned before, every convex polygon $P\in\Pm_2$  is the Minkowski sum of segments and triangles, thus we can write 
$P=\sum_{i=1}^NT_i$ for some $N\in\N$ and $T_i\in\Pm_{2,3}$. By additivity, this implies that $\phi(P)$ lies in the
cone $C\subset\R^k$ generated by $\phi(T_1),\dots, \phi(T_N)$. By the Carath\'eodory theorem for cones, there exist $I\subseteq[N]$ of
size $|I|\leq k$ and $\{\alpha_i \,:\, i\in I\}\subset\R_{> 0}$, such that
$$\phi(P)=\sum_{i\in I}\alpha_i\phi(T_i)=\phi\Big(\sum_{i\in I}\alpha_iT_i\Big).$$ 
As the polytope $\sum_{i\in I}\alpha_iT_i$ has at most $3k$ vertices, we see that $\phi(P)\in\phi(\Pm_{2,3k})$, as claimed.

It remains to show $\overline{\phi(\Pm_{2,3k})}=\overline{\phi(\Km_2)}$. The containment $\overline{\phi(\Pm_{2,3k})}\subseteq\overline{\phi(\Km_2)}$
is clear, as $\Pm_{2,3k}\subset\Km_2$. Conversely, since $\Km_2=\overline{\Pm_2}$ and $\phi$ is continuous, we have
$$\phi(\Km_2)=\phi(\overline{\Pm_2})\subseteq\overline{\phi(\Pm_2)}=\overline{\phi(\Pm_{2,3k})},$$
which implies $\overline{\phi(\Km_2)}\subseteq\overline{\phi(\Pm_{2,3k})}$.
\end{proof}

Recall that a map between two semialgebraic sets is called semialgebraic if its graph is a semialgebraic set, see \cite[Sec 2]{BCR}. The above proposition shows that to understand semialgebraicity of an additive continuous map
$\phi:\Km_2\to\R^k$ it is enough to consider its restriction to the finite-dimensional subspace  $\Pm_{2,m}\subset\Km_2$. Note that the space $\Pm_{2,m}$ is parametrized by $2m$ real numbers, the coordinates of the vertices of polytopes in $\Pm_{2,m}$. It is not hard to see that the area functional $\Vol_2:\Pm_{2,m}\to\R, P\mapsto \Vol_2(P)$ is semialgebraic. Since the mixed area is a composition of Minkowski addition, the area functional, and arithmetic operations (see \re{polarization-2}), it follows that the mixed area functional 
$\V:{\Pm_{2,m}}\times {\Pm_{2,m}} \to\R$, $(P_1,P_2)\mapsto\V(P_1,P_2)$ is also semialgebraic. 
The proofs of these statements, although straightforward, are rather technical, so we put them in Appendix~\ref{appendix}. 



The next theorem is the main result of this section.

\begin{thm}\label{T:semialgebraicity}
For any $n\geq 2$ the closure of the pure mixed volume configuration space $\PMV(n,2)$ is a semialgebraic subset of $\R^{n\choose 2}$.
\end{thm}

\begin{pf}
Let $P=(P_1,\dots, P_n)$ be a tuple of convex polytopes in $\Pm_2$
and $W_P$ the corresponding pure mixed area configuration vector, see \re{pure-config-vector}.
Consider the map $\phi:\Km_2\to\R^{n-1}$ which sends $P_i$ to the $n-1$ mixed areas $(\V(P_i,P_j) \,:\, j\in[n], j\neq i)$.
This map is additive and continuous. By \rp{cone-generators}, there exists $Q_i\in\Pm_{2,m}$ such that
$\phi(K_i)=\phi(Q_i)$, where $m=3(n-1)$. 
Repeating this argument for every $i\in[n]$ we see that there exists $Q=(Q_1,\dots, Q_n)\in{\Pm_{2,m}}^{\!\!\!n}$
such that  $W_P=W_Q$. Let $\psi:{\Km_2}^{\!\!n}\to\R^{n\choose 2}$ be the map defined by $\psi(K)=W_K$, for $K\in{\Km_2}^{\!\!n}$.
We have, thus, shown that $\psi({\Pm_2}^{\!n})=\psi({\Pm_{2,m}}^{\!\!\!n})$.
Similarly to the proof of \rp{cone-generators}, we have $\overline{\psi({\Pm_{2,m}}^{\!\!\!n})}=\overline{\psi({\Km_2}^{\!\!n})}=\overline{\PMV(n,2)}$.
Since the closure of a semialgebraic set is semialgebraic, it remains to check the semialgebraicity of $\psi({\Pm_{2,m}}^{\!\!\!n})$, which
 follows immediately from the semialgebraicity of the mixed area, see \rp{semialg-technical}.
\end{pf}


\begin{thm}\label{T:PMV-8-2} 
For any $n\geq 4$ we have
$$\mathcal{PMV}(n,2)\subseteq \Pl_n:=\left\{v\in\R^{n \choose2} \,:\,  
v_{ij}v_{kl}\leq v_{ik}v_{jl}+v_{il}v_{jk}\ \text{\rm for }\{i,j\}\sqcup\{k,l\}\subseteq[n] \right\}.$$
Moreover, for $n\geq 8$ the above inclusion is proper.
\end{thm}

\begin{proof}
The inclusion follows from the proof of \rt{PMV-4-2}. Suppose $n=8$ and consider the subset $X\subset \R^{8\choose 2}$
given by 
$$X=\left\{v\in \R^{8\choose 2} \,:\, v_{i, i+4}=0, v_{i, l}=v_{i+4, l}, \text{ for }i\in[4],\ l\in[8]\setminus\{i,i+4\}\right\}.$$
We claim that $\dim\left(X\cap\Pl_8\right)=6$ while $\dim\left(X\cap\PMV(8,2)\right)=5$, hence,
$\PMV(8,2)\subsetneq \Pl_8$. Indeed, the former is clear since $X\cap\Pl_8$ can be identified
with $\Pl_4$ (each coordinate of $v\in X$ is either zero or equals $v_{ij}$ for
some $\{i,j\}\subset[4]$; and the Pl\"ucker-type inequalities restricted to $X$ are either trivial or coincide with
the ones for $\Pl_4$). For the latter, note that
the conditions on the mixed areas $\V(K_i, K_{i+4})=0$
imply that $K_i$ and $K_{i+4}$ are parallel segments and the conditions $\V(K_i, K_l)=\V(K_{i+4}, K_l)$
ensure that $K_i$ and $K_{i+4}$ have the same lengths. Thus, $X\cap\PMV(8,2)$
can be identified with the subset
$$Y=\left\{\V(K_i,K_j)\in \PMV(4,2)\,:\,K_1,\dots, K_4 \text{ are segments}\right\}.$$
But in this case $\V(K_i,K_j)$ 
satisfy the Pl\"ucker relation \re{PR} and, therefore, 
$$\dim\left(X\cap\PMV(8,2)\right)=\dim Y=5.$$

For $n\geq 8$, let $\pi:\R^{n\choose 2}\to\R^{8\choose 2}$ be the projection given by $\pi(v)=(v_{ij}\,:\,\{i,j\}\subset[8])$.
Clearly, $\pi(\PMV(n,2))=\PMV(8,2)$. Also, $\pi(\Pl_n)=\Pl_8$, as $\Pl_8$ can embedded into $\Pl_n$ using the 
map \re{embedding} iterated $n-8$ times. Since $\PMV(8,2)\neq \Pl_8$, it follows that $\PMV(n,2)\neq \Pl_n$ as well.
\end{proof}

\section{Dimensions of ${\MV(n,2)}$ and ${\PMV(n,2)}$}\label{S:dim}

In this section we show that $\MV(n,2)\subset\R^{n+1\choose 2}$ has a nonempty interior and, hence, 
has full topological dimension. Since  $\PMV(n,2)$ is the projection of $\MV(n,2)$ onto ${n\choose 2}$ coordinates, it is full-dimensional as well.

\begin{thm}\label{T:dim} For any $n\geq 2$ the topological dimension of $\MV(n,2)$ equals ${n+1\choose 2}$. Consequently, 
the topological dimension of $\PMV(n,2)$ equals ${n\choose 2}$.
\end{thm}
\begin{proof}  Let  $\R[x_1,\dots,x_n]_2$ be the space of quadratic forms in $n$-variables with real coefficients.
Consider the map $\V:{\Km_2}^{\!\!n}\to\R[x_1,\dots,x_n]_2$ which sends an $n$-tuple
$K=(K_1,\dots,K_n)$ to the polynomial $f_K=\Vol_2(x_1K_1+\dots+x_nK_n)$. 

As we mentioned in the introduction,
all mixed areas $\V(K_i,K_i)$ for $1\leq i\leq j\leq n$ appear in
the coefficients of $f_K$:
\begin{equation}\label{e:quadratic-form}
f_K(x)=\sum_{i=1}^n\Vol_2(K_i)x_i^2\,+\,2\!\!\!\sum_{1\leq i<j\leq n}\!\!\V(K_i,K_j)x_ix_j.
\end{equation}
The idea of the proof is to consider a finite-dimensional subspace $\cL\subset{\Km_2}^{\!\!n}$ such that
the restriction $\V:\cL\to\R[x_1,\dots,x_n]_2$ is locally surjective. 

We start with the $n$-tuple of segments $I=(I_1,\dots, I_n)\in{\Km_2}^{\!\!n}$, where $I_i=[0,e_1+ie_2]$.
Let $\cL$ be the space of $n$-tuples of polytopes which are Minkowski sums of the $I_i$ dilated by some non-negative factors:
$$\cL=\left\{\left(a_{11}I_1+\dots+a_{n1}I_n,\ \ldots, \ a_{1n}I_1+\dots+a_{nn}I_n\right)\,:\,a_{ij}\in\R_{\geq 0}\right\}.$$
It is convenient to use matrix multiplication notation $I A$ for the $n$-tuple above, where $A\in\M(n,\R_{\geq 0})$. Similarly, we denote
$Kx:=x_1K_1+\dots+x_nK_n$.  Then our goal is to show that the map $\V$ which sends $IA$ to $f_{IA}\in\R[x_1,\dots,x_n]_2$ is locally surjective.

Fix a matrix $A\in\GL(n,\R_{>0})$. Note that $f_{IA}(x)=\Vol_2(IAx)=f_{I}(Ax)$. Thus, for any $H\in\M(n,\R)$ we can write using Taylor expansion
$$f_{I}((A+H)x)=f_I(Ax)+\langle \nabla f_I(Ax),Hx\rangle+O(\|Hx\|^2).$$
This shows that the Frech\'et derivative of $\V$ at point $IA$ is the linear map $L_A:H\mapsto \langle \nabla f_I(Ax),Hx\rangle$.
To see that $\V$ is locally surjective near $IA$ we need to check that $L_A$ is surjective.

Indeed, by \re{quadratic-form}, 
$$f_I(x)=2\!\!\!\sum_{1\leq i<j\leq n}\!\!(j-i)x_ix_j=x^{\rm T}Tx,$$ where $T=(|i-j|)_{i,j=1}^n\in\GL(n,\R)$ is a Toeplitz matrix.
Hence, $\nabla f_I(x)=2Tx$ and we can write $L_A(H)=2\langle TAx,Hx\rangle$. Since $TA\in\GL(n,\R)$ we can apply a change
of variables $y=TAx$ so that $L_A(HTA)=2\langle TAx,HTAx\rangle=2\langle y,Hy\rangle$. This shows that in appropriate bases for
$\M(n,\R)$ and $\R[x_1,\dots,x_n]_2$,  the map $L_A$ is simply $H\mapsto 2\langle x,Hx\rangle$, which is clearly surjective.
\end{proof}


%

\section{Application to Algebraic and Tropical Geometry}\label{S:applications}

In this section we discuss direct consequences of our results in the intersection theory of algebraic curves and tropical curves. For all necessary definitions in toric and tropical geometry and interconnection between the two areas 
we refer the reader to introductory articles \cite{RGST05, BB13, BK14} or popular textbooks \cite{Sturm02, MaclaganSturmfels15}.

Prior to formulating the consequences, we shortly introduce the basic notions of tropical geometry. 
The tropical semiring is the set $\T := \R  \cup \{+\infty\}$ endowed with the tropical addition $\oplus := \min$ and the tropical multiplication $\odot := +$, where $+\infty$ is neutral with respect to the tropical addition and thus is the tropical zero. A {\it tropical Laurent polynomial}
is an expression of the form $$f:=\bigoplus_{(a,b) \in S} c_{a,b} \odot x^a \odot y^b,$$
where $S \subseteq \Z^2$ is a finite set and $c_{a,b} \in \R$. The convex hull $\conv(S)\subset\R^2$
is called the Newton polytope of $f$ and is denoted by $P_f$. 

 In the classical arithmetic operations over reals,  $f$ is expressed as the piecewise-linear function
\begin{align} \label{trop:pol:as:min}
	f(x,y) = \min \setcond{ c_{a,b} + a x + b y }{(a,b) \in S}. 
\end{align} 
We define the {\it tropical curve} $V(f) \subset \R^2$  associated to $f$ as the set of all points $(x,y)$, in which  the piecewise linear function $f$ is not differentiable, that is where the minimum in \eqref{trop:pol:as:min} is attained at least twice.
 The set $V(f)$  has the structure of a  one-dimensional polyhedral complex; it has vertices and edges.  The polyhedral complex $V(f)$ is dual to the regular polyhedral subdivision of $P_f$ obtained by lifting $(a,b)$ to height $c_{a,b}$. This subdivision induces weights on vertices and edges of $V(f)$. The weight $\omega(v)$ of a vertex $v$ in $V(f)$ is the normalized area of the respective two-dimensional cell in the regular subdivision. The weight $\omega(e)$ of an edge $e$ in $V(f)$ is the lattice length of the respective edge in the regular subdivision. 

For tropical bivariate polynomials $f$ and $g$, the tropical curves $V(f)$ and $V(g)$ are said to intersect transversally if $V(f) \cap V(g)$ is a finite set and if each intersection point $p \in V(f) \cap V(g)$ lies in the relative interior of an edge $e$ of $V(f)$ and an edge $h$ of $V(g)$. The {\it multiplicity} of  a transversal intersection point $p$ is defined to be  $\omega(e) \omega(h) |\det(u,v)|$, where $u$ and $v$ are primitive vectors in the direction of $e$ and $h$, respectively. The total number of intersection points, counting multiplicities, is called the {\it intersection number} of $V(f)$ and $V(g)$ and is denoted $I(V(f), V(g))$.

The following result connects the intersection number of tropical curves and the mixed area. 

\begin{thm}[Tropical BKK theorem in the plane] \label{T:trop:bkk}
	Let $f, g$ be bivariate tropical polynomials, for which the tropical curves $V(f)$ and $V(g)$ intersect transversally. Then the intersection number
	$I(V(f), V(g))$ is equal to $2 \V(P_f,P_g)$. 
\end{thm}

We now can consider arrangements of $n$ tropical curves $C_1,\ldots,C_n$ whose pairwise intersections are transversal. We call such an arrangement transversal. For $\{i,j\}\subset[n]$,  let $I(C_i,C_j)$ be the intersection number of $C_i$ and  $C_j$. In this way, 
with the arrangement $C = (C_1,\ldots, C_n)$ we associate the configuration vector 
\[
	I_C := \Bigl( I(C_i,C_j) \,:\,  \{i,j\}\subset[n]\Bigr) \in \Z_{\ge 0}^{\binom{n}{2}}
\] 
of the intersection numbers of $C$. Analogously to $\PMV(n,2)$, we can consider the set of all such configuration vectors over all transversal arrangements of $n$ tropical curves: 
\[
	\cI(n,2) := \setcond{ I_C }{ C  \ \text{transversal arrangement of $n$ planar tropical curves} }.
\]

Note that finding a description of $\cI(n,2)$ is a more intricate task than describing $\PMV(n,2)$. Because of the discrete nature of $\cI(n,2)$, it is expected that number-theoretic aspects would be involved when describing $\cI(n,2)$. Nevertheless, $\PMV(n,2)$ is intimately related to $\cI(n,2)$ in the following sense.

We say that a subset $H\subset\R^d$ is {\it positively homogeneous} if $x\in H$ implies $\lambda x\in H$
for all $\lambda\geq 0$. Clearly, for any $X\subset\R^d$ the set
$$\R_{\geq 0}\,X=\{\lambda x\,:\, \lambda\in\R_{\geq 0},\, x\in X\}$$
is the minimal, with respect to inclusion, positively homogeneous set containing $X$. We have the following proposition.

\begin{prop}\label{P:tropical}
The inclusion-minimal positively homogeneous closed set containing $\cI(n,2)$ coincides with the topological closure of $\PMV(n,2)$. In other words,
$$\overline{\R_{\geq 0}\,\cI(n,2)}=\overline{\PMV(n,2)}.$$
\end{prop} 
\begin{proof} 
	By \rt{trop:bkk}, 
	\[
		\cI(n,2) = \setcond{ 2\V_P }{P = (P_1,\ldots,P_n) \ \text{is an $n$-tuple of planar lattice polytopes}},
	\] 
	where the polytope can be a polygon, a segment, or a point.
	This implies $\cI(n,2) \subseteq \PMV(n,2)$ and so $\overline{\R_{\geq 0}\,\cI(n,2)}\subseteq\overline{\PMV(n,2)}$.
	
	Conversely, note that $\R_{\geq 0}\,\cI(n,2)$ contains the set
	\[
			\setcond{ \V_P }{P = (P_1,\ldots,P_n) \ \text{is an $n$-tuple of planar rational polytopes}}.
	\]
	Since the set of all planar rational polytopes is dense in the set of all nonempty planar compact convex sets and since the mixed area is continuous  in each argument with respect to the Hausdorff metric, we see
	that $\overline{\PMV(n,2)}\subseteq \overline{\R_{\geq 0}\,\cI(n,2)}$.
	
\end{proof} 

\begin{cor} \label{C:I:4-2} The Pl\"ucker space
	\[
		\Pl_4=\left\{v\in\R_{\geq 0}^{4\choose 2}  \,: \, 
		v_{ij}v_{kl}\leq v_{ik}v_{jl}+v_{il}v_{jk}\ \text{\rm for }\{i,j\}\sqcup\{k,l\}=[4] \right\}
	\]
	is the inclusion-minimal positively homogeneous closed set containing $\cI(4,2)$.
	\end{cor} 
\begin{proof} 
	This is a direct corollary of \rt{PMV-4-2} and \rp{tropical}. 
\end{proof} 
One can interpret this corollary as follows: The closure of the subset of $\R_{\geq 0}^6$ described by all (infinitely many) homogeneous polynomial inequalities valid for the six intersection numbers of any quadruple of tropical curves, is a semialgebraic set determined by the three Pl\"ucker type inequalities~\re{PTE-4}.

The application in toric geometry is completely analogous. A Laurent polynomial $f \in \C[x^{\pm 1},y^{\pm 1}]$ of the form 
$$f = \sum_{(a,b) \in S} c_{a,b} x^a y^b,\ \ \text{with}\  c_{a,b} \in \C\setminus \{0\}$$ defines the curve 
$V(f) := \setcond{ (x,y) \in (\C \setminus \{0\})^2}{f(x,y) = 0}$ in the algebraic torus $(\C \setminus \{0\})^2$. As before, the finite subset $S\subset\Z^2$ 
is the support of $f$ and its convex hull $P_f:=\conv(S)$ is the Newton polytope of $f$.

The next theorem is a 2-dimensional case of the BKK theorem for the number of solutions of generic polynomial systems with fixed supports.
In order to make the notion ``generic'' precise, we recall the definition of a {\it facial subsystem} in our situation.
Let $f$ be a bivariate Laurent polynomial with Newton polytope $P_f$. Given $u\in\Ss^1$, let
$P_f^{u}=\max\{\langle u, v\rangle\,:\,v \in P_f\}$ be the corresponding face of $P_f$. Furthermore, let $f^u$ be the Laurent
polynomial consisting of those terms of $f$ whose support lies in $P_f^{u}$
$$f ^u= \sum_{(a,b) \in S\cap P_f^u} c_{a,b} x^a y^b.$$
If $f,g$ are two Laurent polynomials then
any $u\in\Ss^1$ defines a facial subsystem $f^u=g^u=0$.

\begin{thm}[BKK theorem in the plane]\label{T:BKK-toric}
	Let $f, g \in \C[x^{\pm 1},y^{\pm 1}]$ be Laurent polynomials with Newton polytopes $P_f$ and $P_g$ such that
	for any  $u\in\Ss^1$ the facial subsystem $f^u=g^u=0$ has no solutions in $(\C \setminus \{0\})^2$.
	Then $V(f) \cap V(g)$ is a subset of $(\C \setminus \{0\})^2$ consisting of exactly $2 \V(P_f,P_g)$ points, counting multiplicities. 
\end{thm} 

The condition of \rt{BKK-toric} is called the BKK condition. Note that using an appropriate
monomial change of variables each facial subsystem $f^u=g^u=0$ can be turned to a univariate system. It has a solution
if and only if the resultant of the two polynomials vanishes. Therefore, the BKK condition is satisfied 
for any $f,g$ whose coefficients lie outside of an algebraic hypersurface in the coefficient space.

\rt{BKK-toric} provides yet another interpretation of $\cI(n,2)$:
\[
	\cI(n,2) := \setcond{ | V(f) \cap V(g) | }{f, g \in \C[x^{\pm 1},y^{\pm 1}] \ \text{satisfy the BKK condition} }. 
\]
In light of this interpretation, Corollary~\ref{C:I:4-2} can be viewed as a toric geometry result.

We end with an example illustrating  \rc{I:4-2}.

\begin{ex} 
Consider four tropical curves given by tropical Laurent polynomials $f_1,\dots,f_4$ with
Newton polytopes 
\begin{align*}
P_1&=\conv\{0,e_1,e_2\},\quad\quad\quad\quad\ \, P_2=\conv\{0,2e_1,2e_2,2e_1+e_2,e_1+2e_2\}\\
P_3&=\conv\{0,e_1,e_2,e_1+e_2\},\quad P_4=\conv\{0,3e_1,2e_2,3e_1+e_2,2e_1+2e_2\}.
\end{align*}
We choose the support $S_i$ to be the set of lattice points in $P_i$, for $i\in[4]$, and define the heights according to the lifting function
$h_i: (a,b)\mapsto c_{a,b}$, where we choose 
\begin{align*}
h_1&=a^2+b^2+(a+b)^2,\quad\quad\quad\quad\quad\  h_2=a^2+b^2+(a+b)^2+a-2b\\
h_3&=a^2+b^2+(a+b)^2-a+b,\quad\quad  h_4=a^2+b^2+(a+b)^2+2a-7b.
\end{align*}
We depict the configuration of tropical curves $C=(C_1,\dots, C_4)$, where $C_i=V(f_i)$,  
in \rf{tropical-curves}. 
\begin{figure}[h]
\begin{center}
\includegraphics[scale=.85]{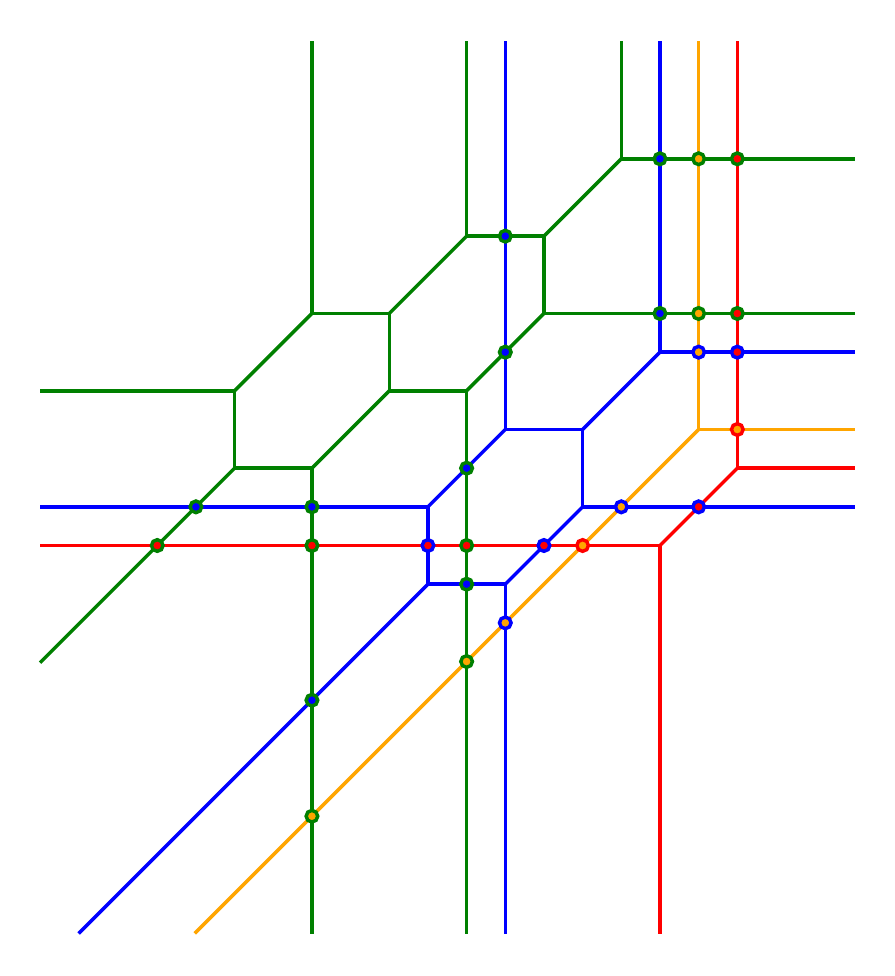}
\end{center}
\caption{A transversal arrangement of four tropical curves and their intersections}
\label{F:tropical-curves}
\end{figure}
We use orange for $C_1$, blue for $C_2$, red for $C_3$, and green for $C_4$.
Note that all intersections are transversal, all
edges have weight one, and all local multiplicities are one as well. 
We obtain the intersection vector
$$I_C=(I_{12},I_{13},I_{14},I_{23},I_{24},I_{34})=(3,2,4,4,9,5),$$
where $I_{ij}=I(C_i,C_j)$. We see that the three products 
$$I_{12}I_{34}=15,\quad I_{13}I_{24}=18,\quad I_{14}I_{23}=16$$
satisfy the three Pl\"ucker-type (triangle) inequalities. 
The Newton polytopes $P_1,\dots, P_n$ with the corresponding regular subdivisions are shown in \rf{subdivisions}.
The intersection number $I_{ij}$, which coincides with the normalized mixed area $2\V(P_i,P_j)$, labels the edge
between $P_i$ and $P_j$, for all $\{i,j\}\subset[4]$.
\begin{figure}[h]
\begin{center}
\includegraphics[scale=.3]{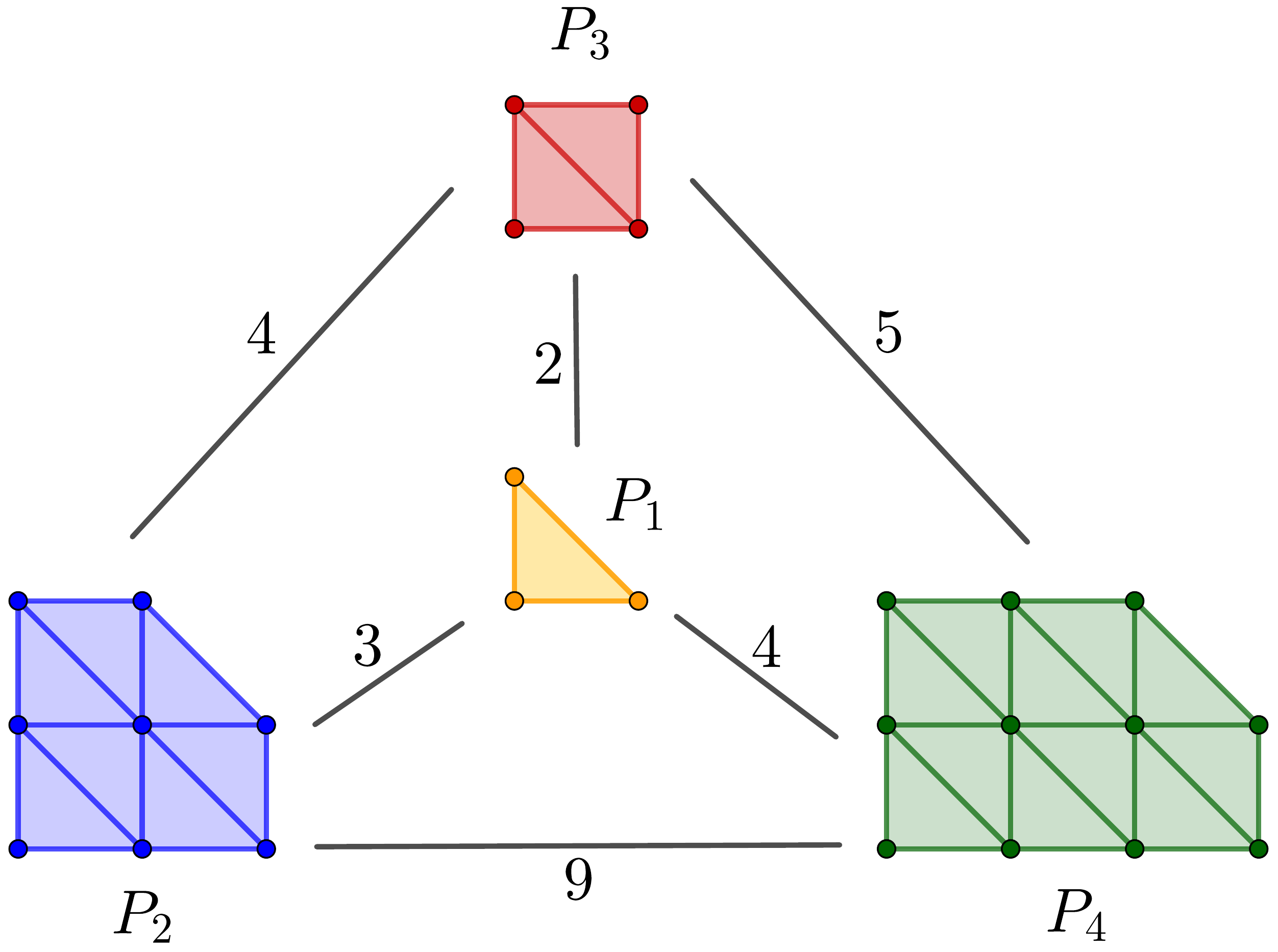}
\end{center}
\caption{Four regular subdivisions of Newton polytopes and the six normalized mixed areas}
\label{F:subdivisions}
\end{figure}
\end{ex}

\section{Appendix}\label{appendix}

In this appendix we give a formal proof of the semialgebraicity of the area and mixed area functionals.

Recall that a {\it first-order formula (over reals)} is a formula obtained from a finite set of polynomial inequalities $f(x)\geq 0$ for $f\in\R[x_1,\dots,x_n]$ by applying a finite number of negations $\neg$, conjunctions $\wedge$, disjunctions $\lor$, and quantifiers $\exists$, $\forall$ on the variables. See \cite[Sec 2]{BCR} for details.

\begin{prop}\label{P:semialg-technical} Let $s,t\in\N$. Consider the following conditions and functions depending on
$p_1,\dots,p_s,q_1,\dots,q_t\in\R^2$:
\begin{enumerate}
\item[(a)] $\text{\rm OrdVertSet}(p_1,\dots,p_s):=$``$p_1,\dots,p_s$ are vertices of a convex $s$-gon listed
in an counterclockwise order'',
\item[(b)] $\text{\rm EqConv}(p_1,\dots,p_s;q_1,\dots,q_t):=$``$\conv(p_1,\dots,p_s)$ equals $\conv(q_1,\dots,q_t)$'',
\item[(c)] $\text{\rm VolConv}(p_1,\dots,p_s):=\Vol_2(\conv(p_1,\dots,p_s))$,
\item[(d)] $\text{\rm VolSumConv}(p_1,\dots,p_s;q_1,\dots,q_t):=\Vol_2(\conv(p_1,\dots,p_s)+\conv(q_1,\dots,q_t))$,
\item[(e)] $\text{\rm MixedArea}(p_1,\dots,p_s;q_1,\dots,q_t):=\V(\conv(p_1,\dots,p_s),\conv(q_1,\dots,q_t))$.

\end{enumerate}
Then the conditions (a) and (b) can be expressed by first-order formulas and the functions (c), (d) and (e) are semialgebraic.
\end{prop}

\begin{pf} (a) As the cases $s=1,2$ are trivial, so we assume $s\geq 3$. Then we have the following first-order formula 
for $\text{\rm OrdVertSet}(p_1,\dots,p_s)$:
$$\text{\rm OrdVertSet}(p_1,\dots,p_s)= \bigwedge_{i=1}^s\bigwedge_{j\in[s]\setminus\{i,i+1\}}\det(p_{i+1}-p_{i},p_j-p_i)>0,$$
where the indices are considered modulo $s$.

(b) Note that $\conv(p_1,\dots,p_s)$ equals $\conv(q_1,\dots,q_t)$ if and only if $p_1,\dots,p_s$ lie in
$\conv(q_1,\dots,q_t)$ and $q_1,\dots q_t$ lie in $\conv(p_1,\dots,p_s)$. Therefore we can express
$\text{\rm EqConv}(p_1,\dots,p_s;q_1,\dots,q_t)$ as the conjunction of 
$$\bigwedge_{i=1}^s\left(\exists \lambda_{i_1},\dots,\exists\lambda_{i_t}\in\R_{\geq 0} : (\lambda_{i_1}+\dots+\lambda_{i_t}=1)\wedge(p_i=\lambda_{i_1}q_1+\dots+\lambda_{i_t}q_t) \right)$$
and
$$\bigwedge_{j=1}^t\left(\exists \mu_{j_1},\dots,\exists\mu_{j_s}\in\R_{\geq 0} : (\mu_{j_1}+\dots+\mu_{j_s}=1)\wedge(q_j=\mu_{j_1}p_1+\dots+\mu_{j_s}p_s) \right),$$
which is a first-order formula.

(c) Let $P=\conv(p_1,\dots,p_s)$ and assume that $q_1,\dots,q_t$ are the vertices of $P$ in the counterclockwise order, in particular, $\text{\rm OrdVertSet}(q_1,\dots,q_t)$ is true. Then $\Vol_2(P)=0$ for  $t=1,2$ and
$$
\Vol_2(P)=\frac{1}{2}\sum_{i=2}^{t-1}\det(q_{i+1}-q_1,q_i-q_1), \text{ for }t\geq 3,
$$
which is clearly a polynomial in the coordinates of the $q_i$. Now, 
since $P$ has at most $s$ vertices, the statement
$$\exists q_1\in\R^2,\dots, \exists q_s\in\R^2\,:\,\bigvee_{t=1}^s\left(\text{\rm OrdVertSet}(q_1,\dots,q_t)\wedge\text{\rm EqConv}(p_1,\dots,p_s;q_1,\dots,q_t)\right)$$
is true. Therefore, for $V\in\R_{\geq 0}$ we can express the condition $V=\Vol_2(P)$ as
\begin{align*}
& \exists q_1\in\R^2,\dots, \exists q_s\in\R^2\,:\, \\
& \bigvee_{t=1}^s\Big(\text{\rm OrdVertSet}(q_1,\dots,q_t)\wedge\text{\rm EqConv}(p_1,\dots,p_s;q_1,\dots,q_t)
\wedge \big(V=\frac{1}{2}\sum_{i=2}^{t-1}\det(q_{i+1}-q_1,q_i-q_1)\big)\Big),
\end{align*}
which is a first-order formula. This shows that the graph of the function $\text{\rm VolConv}(p_1,\dots,p_s)$
is semialgebraic, i.e. it is a semialgebraic function.

(d) We have $\conv(p_1,\dots,p_s)+\conv(q_1,\dots,q_t)=\conv(p_i+q_j : 1\leq i\leq s, 1\leq j\leq t)$. Thus, 
the function $\text{\rm VolSumConv}$ is the composition of the linear map 
$$\phi:(\R^{2})^{s+t}\to (\R^{2})^{st}, \quad (p_1,\dots, p_s,q_1,\dots, q_t)\mapsto (p_i+q_j : 1\leq i\leq s, 1\leq j\leq t)$$
and the function $\text{\rm VolConv}$ from part (c). Therefore, the semialgebraicity of $\text{\rm VolSumConv}$ 
follows from that of $\text{\rm VolConv}$.

(e) This immediately follows from parts (c) and (d) and the formula \re{polarization-2} for the mixed area
$$\V(P_1,P_2)=\frac{1}{2}\left(\Vol_2(P_1+P_2)-\Vol_2(P_1)-\Vol_2(P_2)\right).$$

\end{pf}

We remark that with some additional work the proofs can be extended to the case of $d$-dimensional volume and mixed volume, but
it is not needed for the purposes of this paper.

\bibliographystyle{amsalpha}
\bibliography{lit}

\end{document}